\documentclass[12pt, reqno]{amsart}
\usepackage[margin=0.5in]{geometry}
\usepackage{tikz}
\newif\ifdraft
\drafttrue

\usepackage[foot]{amsaddr}
\usepackage{fixltx2e,amssymb,indentfirst,xspace,bm}
\usepackage{framed,paralist,color}
\usepackage{multicol}
\usepackage[off]{pdfsync}
\usepackage{mathtools}
\usepackage[normalem]{ulem}
\usepackage[initials,nobysame]{amsrefs}
\usepackage{lipsum}
\makeatletter
\renewcommand{\email}[2][]{%
  \ifx\emails\@empty\relax\else{\g@addto@macro\emails{,\space}}\fi%
  \@ifnotempty{#1}{\g@addto@macro\emails{\textrm{(#1)}\space}}%
  \g@addto@macro\emails{#2}%
}
\makeatother

\makeatletter
\def\@wraptoccontribs#1#2{}

\@mparswitchfalse
\makeatother
\newcounter{marnote}

\ifdraft
    \newcommand{\sidenote}[1]{\pdfsyncstop\marginpar[\raggedleft\tiny #1]{\raggedright\tiny #1}\pdfsyncstart}

\else
    \newcommand{\sidenote}[1]{}
    
\fi
\newcommand{\warning}[1]{\typeout{}\typeout{WARNING: #1 at line \the\inputlineno}\typeout{}}
\newenvironment{note}[1][TODO. ]{%
    \ifdraft\else\warning{Note environment still present in final version}\fi
    \MakeFramed{\advance\hsize-\width \FrameRestore}\noindent\textbf{#1}}%
    {\endMakeFramed}

\makeatletter
\newcommand{\UWave}[2][blue]{\bgroup \markoverwith{\textcolor{#1}{\lower3.5\p@\hbox{\sixly \char58}}}\ULon{#2}}
\makeatother
\newcommand{\SOut}[2][red]{\bgroup\markoverwith {\textcolor{#1}{\rule[.45ex]{2pt}{.1ex}}}\ULon{#2}}

\newcommand{\highlight}[2][yellow]{\bgroup\markoverwith {\textcolor{#1}{\rule[-.2em]{2pt}{1.2em}}}\ULon{#2}}
\newcommand{\trc}{\mathrm{tr}}
\newcommand{\vn}{v^{(n)}}
\newcommand{\Qn}{Q^{(n)}}
\newcommand{\An}{A^{(n)}}
\newcommand{\Omn}{\Omega^{(n)}}

\newcommand{\dT}{\delta T}
    {\left\{\begin{IEEEeqnarraybox}[\relax][c]{#1}}%
    {\end{IEEEeqnarraybox}\right.}
\def\be{\begin{equation}}
\def\ee{\end{equation}}

\def\bea#1\eea{\begin{align}#1\end{align}}

\def\non{\nonumber}

\newcommand{\delt}{\prt_t}

\newcommand{\R}{\mathbb{R}}

\renewcommand{\epsilon}{\varepsilon}
\renewcommand{\leq}{\leqslant}
\renewcommand{\geq}{\geqslant}

\newif\iftextstyle
\textstyletrue
\everydisplay\expandafter{\the\everydisplay\textstylefalse}

\numberwithin{equation}{section}
\newtheorem{theorem}{Theorem}[section]

\newtheorem{lemma}[theorem]{Lemma}
\newtheorem{proposition}[theorem]{Proposition}

\newtheorem*{theorem*}{Theorem}
\newtheorem*{lemma*}{Lemma}
\newtheorem*{proposition*}{Proposition}
\newtheorem*{corollary*}{Corollary}

\theoremstyle{definition}

\theoremstyle{remark}
\newtheorem{remark}[theorem]{Remark}
\newtheorem*{remark*}{Remark}

\newtheoremstyle{cases}% name
  {\smallskipamount}%      Space above, empty = `usual value'
  {\smallskipamount}%      Space below
  {}% Body font
  {\parindent}%         Indent amount (empty = no indent, \parindent = para indent)
  {\itshape}% Thm head font
  {.}%        Punctuation after thm head
  {.5em}%     Space after thm head: " " = normal interword space;
  {\thmnote{#1 #2: #3}}% Thm head spec

\theoremstyle{cases}
\newcommand{\Ignore}[1]{}

\newcommand{\prt}{\ensuremath{\partial}}

\newcommand{\Cal}[1]{\ensuremath{\mathcal{#1}}}

\newcommand{\LP}{\ensuremath{\Cal{P}}} %

\newcommand{\dd}{\mathrm{d}}

\def\RR{{\mathbb R}}

\def\mcF{{\mycal F}}
\def\mcI{{\mycal I}}
\def\mcL{{\mycal L}}
\def\mcN{{\mycal N}}

\def\eps{{\varepsilon}}

\usepackage{etoolbox}
\patchcmd{\section}{\scshape}{\bfseries}{}{}
\makeatletter
\renewcommand{\@secnumfont}{\bfseries}
\makeatother
\newcommand\testname{\textbf{Abstract}}
\makeatletter
\newenvironment{abst}{%
    \small
    \begin{center}%
        {\textsc \testname\vspace{-.2em}\vspace{\z@}}%
    \end{center}%
    \quote
    }
   {\endquote}
\makeatother

\def\bproof{\noindent{\bf Proof.\;}}
\def\eproof{\hfill$\square$\medskip}

\DeclareFontFamily{OT1}{rsfs}{}
\DeclareFontShape{OT1}{rsfs}{m}{n}{ <-7> rsfs5 <7-10> rsfs7 <10-> rsfs10}{}
\DeclareMathAlphabet{\mycal}{OT1}{rsfs}{m}{n}

\def\Id{{\rm Id}}

\def\tr{{\rm tr}}

\def\be{\begin{equation}}
\def\ee{\end{equation}}

\def\Id{{\rm Id}}

\def\tr{{\rm tr}}

\def\be{\begin{equation}}
\def\ee{\end{equation}}
\def\bea#1\eea{\begin{align}#1\end{align}}
\def\non{\nonumber}

\usepackage[affil-it]{authblk}
\usepackage{titling}
\usepackage{fancyhdr}

\numberwithin{equation}{section}

 \newcommand\blfootnote[1]{%
  \begingroup
  \renewcommand\thefootnote{}\footnote{#1}%
  \addtocounter{footnote}{-1}%
  \endgroup
} 
 
\markleft{Francesco De Anna $\quad$ Arghir  Zarnescu}
\keywords {nematic liquid crystal fluids, Navier-Stokes equations, global wellposedness.}
\title{\large 
	\textbf{\uppercase{Global well-posedness \\ and twist-wave solutions\\ for the  inertial Qian-Sheng model\\of liquid crystals}}}
\author{ Francesco De Anna$^1$ and Arghir Zarnescu$^{2,3,4}$ }

\date{\vspace{0.2cm}\small\today}

\begin{document}

\maketitle

\begin{abst}
	We consider the inertial Qian-Sheng model of liquid crystals which couples  a hyperbolic-type equation  involving a second-order material derivative with  a forced incompressible Navier-Stokes system. We study the energy law and prove a global well-posedness result. We further provide an example of twist-wave solutions, that is solutions of the coupled system for which the flow vanishes for all times.
\end{abst}
\blfootnote{\tiny{\\(1) Department of Mathematics, The Pennsylvania State University, University Park,
PA 16802, USA\, {\it Email: fzd16@psu.edu}\\
\hspace{1cm} (2) IKERBASQUE, Basque Foundation for Science, Maria Diaz de Haro 3,
48013, Bilbao, Bizkaia, Spain \\
(3) BCAM, Basque Center for Applied Mathematics, Mazarredo 14, 48009 Bilbao,  Spain\, {\it   Email: azarnescu@bcamath.org}\\
(4)``Simion Stoilow" Institute of Mathematics of the Romanian Academy, 21 Calea Grivitei, 010702 Bucharest, Romania}}

\section{Introduction}
The main aim of this article is to study a system describing the hydrodynamics of nematic liquid crystals in the Q-tensor framework (see for an introduction to the $Q$-tensor framework \cite{NewtonMottram},\cite{ZarIntro}) There exists several such models and we will consider the one proposed by T. Qian and P. Sheng in \cite{QianSheng}. As most tensorial models, this one provides  an extension of the classical Ericksen-Leslie model \cite{Leslie}, in particular capturing the biaxial alignment of the molecules, a feature not available in the classical Ericksen-Leslie model.

 Our main interest in this model is due to the fact that  it incorporates systematically a certain term that models inertial effects. Details about the physical relevance of this  will be provided in the  Subsection~\ref{subsec:physics}, below.

\bigskip 
 The inertial term  is usually neglected on physical grounds, a fact that is also convenient mathematically since keeping it generates considerable analytical and numerical challenges.  From a mathematical point of view the system couples a forced incompressible Navier-Stokes system,  modelling the flow, with a hyperbolic convection-diffusion system for matrix-valued functions that model the evolution of the orientations of the nematic molecules. The inertial term is responsible for the hyperbolic character of the equation describing the orientation of the molecules. This feature is also present in the Ericksen-Leslie model but it is usually neglected in the mathematical studies due to the formidable difficulties in treating it in the presence of the specific unit-length constraint. One can regard our study as a step towards the analytical understanding of the inertial Ericksen-Leslie model where one discards the unit-length constraint. 
 
\smallskip
In order to clearly describe the system it is convenient to introduce some terminology. The local orientation of the molecules is described through a function $Q$ taking values from $\Omega\subset \RR^d$, into the set of the so-called $d$-dimensional $Q$-tensors, that is symmetric and traceless $d\times d$ matrices:
\begin{equation*}
	S_0^{(d)}:=\left\{Q \in \R^{d\times d}; Q_{ij}=Q_{ji},\textrm{tr}(Q) = 0, i,j=1,\dots,d \right\}
\end{equation*}
The evolution of the $Q's$ is driven by the free energy of the molecules, as well as the transport, distortion and alignment effects caused by the flow.

\smallskip
The velocity of the centres of masses of molecules obeys a forced  incompressible Navier-Stokes system, with an additional stress tensor, a forcing term modelling the effect that the interaction of the molecules has on the dynamics of their centres of masses.
Explicitly the equations, in non-dimensional form, are:

\bea\label{eqinv:momentum+}
\dot v+\nabla p-\frac{\beta_4}{2} \Delta v=&\nabla\cdot \bigg(-L\nabla Q\odot \nabla Q+\beta_1 Q \trc\{Q A\}+\beta_5 AQ+\beta_6 QA\bigg)\non\\
&+\nabla\cdot\bigg(\frac{\mu_2}{2}(\dot Q-[\Omega, Q])+\mu_1\, \big[Q,(\dot Q-[\Omega, Q])\big]\bigg)
\eea 
\be\label{vdivfree}
\nabla\cdot v=0
\ee
\be\label{eqinv:tensors+}
J\ddot Q+\mu_1\dot Q=L\Delta Q-aQ+b(Q^2-\frac{1}{d}|Q|^2I_d)-cQ|Q|^2+\frac{\tilde\mu_2}{2} A+\mu_1 [\Omega,Q]
\ee where $\dot f=(\partial_t +v\cdot\nabla )f$ denotes a {\it material derivative} and for any two $d\times d$ matrices $M$, $N$, we denote their commutator as  $[M,N]:=MN-NM$. Furthermore, we denote $A_{ij}:=\frac{1}{2}(v_{i,j}+v_{j,i})$, $\Omega_{ij}:=\frac{1}{2}(v_{i,j}-v_{j,i})$, for $i,j=1,\dots,d$, $(\nabla Q\odot \nabla Q)_{ij}:= \sum_{k,l=1}^d Q_{kl,i}Q_{kl,j}$ (where for a scalar function $f$, we write  $f_{,j}$ for $\frac{\partial f}{\partial x_j}$) and $|Q|=\sqrt{\textrm{tr}(Q^2)}$. The $I_d$ denotes the $d\times d$ identity matrix.

The physical relevance of the equations and their meaning is provided in the next subsection, which can be skipped without impeding on the understanding of the remaining mathematical aspects of the paper.

\subsection{Physical aspects}\label{subsec:physics}
In the following we consider just the $d=3$ case ( out of which  one can reduce everything in a standard manner to the $d=2$ case) and take the domain $\Omega$ to be $\R^3$. The velocity $v$ of the centres of masses the molecules satisfies  a convection-diffusion fluid-type equation, with forcing provided by the pressure $p$, the distortion stress $\sigma$ and the viscous stress $\sigma'$ (here and in the following we use the Einstein summation convention, of summation over repeated indices):
\be\label{eqabst:momentum}
\dot v_i=(-p\delta_{ij}+\sigma_{ij}+\sigma'_{ij})_{,j},
\ee where $p$ is the pressure. 

The fluid is  taken to be incompressible so we have the divergence-free constraint:
\be\label{eq:incompres}
 v_{k,k}=0.
\ee

The distortion stress $\sigma$ is given by
\begin{equation*}
	\sigma_{ij}:=-\frac{\partial\mcF}{\partial( Q_{\alpha\beta,i})} Q_{\alpha\beta,j}
\end{equation*} where we use the simplest form of the Landau-de Gennes free energy density 
\begin{equation*}
	\mcF[Q]:=\frac{L}{2}|\nabla Q|^2+\psi_B(Q)
\end{equation*} modelling the spatial variations through the $\frac{L}{2}|\nabla Q|^2$ term with positive diffusion coefficient $L>0$, and  the nematic ordering  enforced through the ``bulk term" taken to be of the standard form \cite{NewtonMottram}

\be\label{def:psib}
\psi_B(Q)=\frac{a}{2}\tr(Q^2)-\frac{b}{3}\tr(Q^3)+\frac c4 (\tr(Q^2))^2.
\ee

The viscous stress $\sigma'$ is given by\footnote{ Note that in \cite{QianSheng}  the divergence of a matrix is taken along columns rather than rows as we do in here. However we changed everything consistently to fit our definition of matrix divergence}:
\begin{equation}
\begin{aligned}
\sigma'_{ij}:=\beta_1 Q_{ij}Q_{lk}A_{lk}+\beta_4 A_{ij}
								&+\beta_5 Q_{jl}A_{li}+\beta_6 Q_{il}A_{lj}\non\\
								&+\frac{1}{2}\mu_2\mcN_{ij}+\mu_1 Q_{il}\mcN_{lj}-\mu_1 \mcN_{il}Q_{lj},\label{defsigmaprime}
\end{aligned}  
\end{equation}
where $\beta_1,\beta_4,\beta_5,\beta_6$, $\mu_1$ and $\mu_2$ are viscosity coefficients, $A$ is the rate-of-strain tensor
defined by means of
\begin{equation*}
	A_{ij}\:=\frac{v_{i,j}+v_{j,i}}{2},
\end{equation*}
i.e. the symmetric part of the velocity gradient, and $\mcN$ stands for the co-rotational time flux of $Q$, whose $(i,j)$-th component is defined as follows
\begin{equation*}
	\mcN_{ij}:=
	\big(
		\dot Q - \omega \wedge Q + Q\wedge \omega
	\big)_{ij} =
	\delt Q_{ij}+v_kQ_{ij,k}-\eps_{ikl}\omega_k Q_{lj}-\eps_{jkl}\omega_k Q_{il}.
\end{equation*}
$\mcN$ represents the time rate of change of $Q_{ij}$ with respect to the background fluid angular velocity $\omega=\frac{1}{2}\nabla\times v$. Moreover, one can reformulate $\mcN$ making use of the vorticity tensor $\Omega$
\begin{equation*}
	\Omega_{ij}:=\frac{v_{i,j}-v_{j,i}}{2}.
\end{equation*}
Indeed, one can check that 
\begin{equation*}
	\mcN_{ij}=\big( \dot Q - [\Omega,Q]\big)_{ij} =\dot Q_{ij}-\Omega_{il}Q_{lj}+Q_{il}\Omega_{lj},
\end{equation*}
since we have $\omega\times u = \Omega u$, for any $d$-dimensional vector $u$.

\bigskip

\par   We will assume that the viscosity coefficients satisfy the following two constraints 

\begin{equation}\label{parodi-corot}
\begin{alignedat}{4}
	\beta_6\,&-\,\beta_5\,&&=\,\mu_2,\,\\
	\beta_5\,&+\,\beta_6\,&&=\,0.
\end{alignedat}
\end{equation}
For a better understanding of the relation between these conditions and other ones available in the  literature, it is worth making a comparison  between the stress tensor $\sigma$ given in \eqref{defsigmaprime} and the better-known Leslie stress tensor. Indeed whenever $Q$ is uniaxial (with unitary scalar order parameter, for simplicity) i.e $Q(t,x)\,=s_*(n(t,x)\otimes n(t,x)-\frac{1}{3})$ , with $n(t,x)\in\mathbb{S}^2$ the director field, $s_*\not=0$, the tensor $\sigma$ becomes 
%\begin{equation}\label{Leslie-stress}
%\begin{aligned}
%    \sigma_{ij}\,=\,\,\,2\,\beta_1 \delta_{ij}n_kn_lA_{kl}\,+\,\beta_1n_in_jn_kn_lA_{lk}\,+\,(\beta_1\,+\,\beta_4\,+\beta_5\,+\,\beta_6) A_{ij}\,+\beta_5\,n_jn_kA_{ki}\,+\,\\
%    +\,\beta_6n_in_kA_{kj}\,+\,(\,\frac{\mu_2}{2}\,+\,\mu_1)\,
%        n\,\otimes \tilde{\mathcal{N}}\,+\,
%        (\,\frac{\mu_2}{2}\,-\,\mu_1)
%        \tilde{\mathcal{N}}\otimes n,
%\end{aligned}
%\end{equation}
%where we have denoted by $\tilde{\mathcal{N}}\,:=\,\dot n\,-\,\Omega\,n$ the co-rotational time derivative associated to the director field $n$.  The reader can observe that the above formulation of $\sigma$ coincides with the
the better-known Leslie stress tensor of the Ericksen-Leslie theory, up to some relations involving the viscosity coefficients (see \cite{QianSheng}).

% It is worth noting that the first term on the right-hand side is usually neglected in the Leslie stress tensor since it can be absorbed by the definition of the pressure.

\smallskip
The well-known Parodi's relation for the Leslie stress tensor  corresponds in our setting to
\begin{equation}
    \beta_6\,-\,\beta_5\,=\,\mu_2,
\end{equation}
namely the first identity of \eqref{parodi-corot}. 

\smallskip
The second condition in \eqref{parodi-corot} is not always satisfied by physical materials (though for some it is nearly satisfied such as for MBBA, see below) however it is sometimes assumed in the physics  literature in the more specialised form $\beta_5=\beta_6=0$ (see for instance \cite{popa2005waves}).

Moreover, we will need to assume that the Newtonian viscosity $\beta_4$ is large enough compared to the other remaining viscosities, in order to obtain the necessary energy dissipation.

\bigskip

The evolution of the order tensor $Q$ is driven by

\be\label{eqabst:tensors}
J\ddot Q_{ij}=h_{ij}+h'_{ij}-\lambda \delta_{ij}-\eps_{ijk} \lambda_k.
\ee
 where $\eps_{ijk}$, the Levi-Civita symbol. The $\lambda,\lambda_{k}$ are Lagrange multiplier enforcing the tracelessness and symmetry of the tensor and in our case they can be easily determined as $\lambda_k=0$ and
 $\lambda=-b\frac{|Q|^2}{3}I_3$ (with $I_3$ the $3\times 3$ identity matrix).

The elastic molecular field $h$ is
\begin{equation*}
	h_{ij}:=-\frac{\partial \mcF}{\partial Q_{ij}}+\bigg(\frac{\partial \mcF}{\partial(\partial_k Q_{ij,k})}\bigg)_{,k}
\end{equation*} and the viscous molecular field $h'$ is given by:
\be\label{def:hprime}
	h'_{ij}:=\frac 12\tilde\mu_2 A_{ij}-\mu_1 \mcN_{ij},
\ee

\bigskip

The definition of $\tilde \mu_2$ requires some clarifications. We note that in the paper \cite{QianSheng} of Qian and Sheng, the viscosity coefficient $\tilde\mu_2$ corresponds exactly to $\mu_2$  while other authors take it to be 
$\tilde\mu_2\,=\,-\mu_2$, see \cite{popa2003waves, popa2005waves}. The two different choices of the sign for $\tilde\mu_2$ 
provide intrinsical differences at the energy level, as it will be seen in Section~\ref{sec3}. We will see there that it would be more natural to assume $\tilde\mu_2\,=\,-\mu_2$, otherwise a new continuum variable
\begin{equation}\label{def:fakevar}
    \dot Q\,+\,[\Omega,\,Q]
\end{equation}
would effect the time-evolution of the flow. However, if we would take alternatively  $\tilde\mu_2\,=\,-\mu_2$ we would obtain  the classical co-rotational time flux $\mcN\,=\,\dot Q\,-\,[\Omega,\,Q]$ (instead of the above variable in \eqref{def:fakevar}).
%In this work we then impose the following condition on the viscosity coefficient $\mu_2$:
%\begin{equation}\label{mu2cond}
%\left\{\;
%\begin{alignedat}{8}
%    &\mu_2\,=\,0\quad &&\quad    \text{if}\quad\tilde{\mu}_2=\mu_2,\\
%    &\mu_2\in \RR       &&\quad    \text{otherwise}.
%\end{alignedat}
%\right.
%\end{equation}

We assume all the coefficients to be non-dimensional. For a common physical example, the MBBA material, we have the following relations between the coefficients \cite{SvensekZummer}:
\be\label{eq:coeffratios}
\frac{\mu_2}{\mu_1}\sim -1.92,\frac{\beta_1}{\mu_1}\sim 0.17, \frac{\beta_4}{\mu_1}\sim 0.7, \frac{\beta_5}{\mu_1}\sim 0.7, \frac{\beta_6}{\mu_1}\sim -0.79
\ee
Furthermore, because the coefficient $\beta_4$ corresponds to the standard Newtonian stress tensor we can assume
\be\label{eq:beta4}
\beta_4>0
\ee which fixes the signs for all the viscosities.

 The $J$ in \eqref{eqabst:tensors} stands for the inertial density and it is taken to be greater than $0$. This is consistent with the fact that $J$ has the same sign as the inertia in the Leslie-Ericksen type of model (see Appendix B in \cite{QianSheng}) where it is assumed to be positive (see for instance the assumption that J.L. Ericksen makes in \cite{Ericksentwist}).

 \bigskip The inertial term could conceivably play a role when the anisotropic axis is subjected to large accelerations, as motivated by F. Leslie (in the context of the director model) in \cite{leslieacceleration}.

 \smallskip
    Another interesting feature of the inertia is that it captures the wave-like phenomena, and one of the most mysterious and yet simple manifestation of these is related to the so-called twist-waves, introduced by J.L. Ericksen in \cite{Ericksentwist}. These are very special solutions of the coupled system, for which the flow vanishes for all time. The effect of the flow still remains on the $Q$-tensor part, by imposing an additional constraint, so these are very special solutions.

% {\color{red} ADD HERE SCALING}

\bigskip
\subsection{Main results}

We  note first that the system admits a  Lyapunov-type functional, up to some relations on the viscosity coefficients. This functional includes the free energy due to the director field, the kinetic energy of the fluid and most importantly the rotational kinetic energy of the director field. 
{
\smallskip

\begin{theorem}\label{intro:thm1} {\bf [Energy law and apriori control of low-regularity norms]}

 We consider the system \eqref{eqinv:momentum+},\eqref{vdivfree},\eqref{eqinv:tensors+} in $\R^d$, $d=2$ or $d=3$.
	Let us assume that the viscosity coefficients fulfill 
	
	\be
	\beta_1,\beta_4,\mu_1\geq 0,
        \ee	and the inertia coefficient $J$ as well as the diffusion coefficient $L$ are positive.
 Furthermore, we assume:
        
        \begin{equation}\label{parodi-corot+}
\begin{alignedat}{4}
	\beta_6\,&-\,\beta_5\,&&=\,\mu_2,\,\\
	\beta_5\,&+\,\beta_6\,&&=\,0.
\end{alignedat}
\end{equation} 

Concerning $\tilde\mu_2$ we assume that:

\begin{equation}
\label{mu2cond}
\textrm{if }\tilde\mu_2=\mu_2\textrm{ then both of them are set to zero, i.e. }\tilde\mu_2=\mu_2=0.
\end{equation} 
        Moreover, in order to have the free energy of the molecules well defined we assume that the material coefficient $c$ satisfies
       
      \be
      c>0.
      \ee
        
        Then there exists a constant $C_d$ depending on $\tilde\mu_2,\beta_5,\beta_6,\mu_2$ such that if the Newtonian viscosity is large enough, i.e. $\beta_4>C_d$ then for classical solutions  that decay fast enough at infinity\footnote{sufficiently fast to be able to integrate by parts in the proof of the theorem, which happens for instance if they are in the function spaces \eqref{aprioribounds} }    the total  energy decays, i.e.

\be\label{energydecay}
\frac{d}{dt}\int_{\R^d} \frac{1}{2}
		\big(
			|v|^2+J|\dot Q|^2+\frac{L}{2}|\nabla Q|^2
		\big)+\psi_B(Q)\,\dd x\le 0
\ee

	 Furthermore:
	 
	 \begin{itemize}
	 \item If $d=2$ and $a\ge 0$ then $\psi_B(Q)\ge 0$ and for any $T>0$ we have the apriori bounds:	\begin{align}\label{aprioribounds}	
		v&\in L^\infty(0,T; L^2(\RR^d))\cap L^2(0,T;H^1(\RR^d)),\\
		Q&\in L^\infty(0,T;  H^1(\RR^d))\quad\text{with}\quad 	\dot Q \in L^\infty(0,T;L^2(\RR^d)),\non
	\end{align}
	\item If $d=2$ and $a< 0$, or $d=3$ (and $a$ arbitrary) then $\psi_B(Q)$ can be negative\footnote{ and thus the energy decay \eqref{energydecay} does not suffice for providing the apriori bounds \eqref{aprioribounds}}. Then there exists $\bar\mu_1=\bar\mu_1(a,b,c)$, $J_0:=J_0(\bar\mu_1,a,b,c)$,  and $\tilde C_d=\tilde C_d(\tilde\mu_2,\beta_5,\beta_6,\mu_2)>0$  such that if $J<J_0$, $\mu_1>\bar \mu_1$, $\beta_4>\tilde C_d$ then the apriori bounds \eqref{aprioribounds} hold.
	\end{itemize}
	
	\end{theorem}
}

\bigskip
%Let us remark that our restriction on the viscosity coefficients $\beta_1,\,\mu_1,\, \beta_5$ and $\beta_6$ are not unnatural, as one can check in the MBBA example \eqref{eq:coeffratios}. 

{%\color{green}
The proof of Theorem \ref{intro:thm1} exhibits the main characteristic feature of the system, that is the mixing of terms that provide the most suitable cancellation of  ``extraneous'' maximal derivatives, i.e. the highest derivatives in $v$ that appear in the $Q$ equation and the highest derivatives in $Q$ that appear in the $v$ equation. 

%We mainly handle this difficulties by taking into account the specific feature of the coupling that allows for the \textit{cancellation of the worst terms}, when considering certain physically meaningful combination of terms.

\smallskip
It is important to observe that despite these apriori estimates, one cannot expect to construct weak solutions just by making use of this energy relation. Indeed, the most common approach in order to construct weak solutions is the compactness method, i.e. construct approximate solutions (satisfying  similar apriori bounds) and pass to the limit.  In the classical Navier-Stokes equation one needs to take care in dealing with  the nonlinear terms, but the apriori estimates available do provide enough control. However, things are much worse in our system \eqref{eqinv:momentum+}-\eqref{eqinv:tensors+}. The main difficulty  is inside the stress tensor $\sigma_{ij}$, more precisely in the nonlinear term 
%\vspace{-0.3cm}
\begin{equation}\label{intro:nablaQ2}
	(\nabla Q\odot \nabla Q)_{ij} := \sum_{\alpha,\beta=1}^d Q_{\alpha\beta, i}Q_{\alpha\beta, j}  
	%=\frac 12 \frac{\partial |\nabla Q|^2}{\partial Q_{\alpha\beta, i}}Q_{\alpha\beta, j}
	\vspace{-0.1cm}
\end{equation}

Let us note that the estimates provided by the apriori bounds \eqref{aprioribounds} do not suffice for passing to the limit in the divergence of  \eqref{intro:nablaQ2}. This is to be contrasted with the case $J=0$. One should keep in mind that a positive inertial density $J$ leads the order tensor equation to be hyperbolic-like, in contrast to the parabolic structure that occurs when $J$ is neglected. In the parabolic setting one can make use of regularizing effects, achieving a control on two spatial derivatives of $Q$ (i.e. $\Delta Q$), which certainly allows to control the limit of a product as in the divergence of \eqref{intro:nablaQ2}. This feature is lost when $J$ is positive, so that constructing weak solutions would require a different approach than a rather common compactness one based on estimates \eqref{aprioribounds}.

\bigskip Thus one can attempt to construct strong solutions and it will turn out that this can be done. The most interesting aim is then to construct {\it global in time} solutions. We have been able to obtain them, using the one of  main features of the system namely the damping provided by the $\mu_1$ term. Indeed, if one formally takes the flow $v$ to be zero in \eqref{eqinv:tensors+} then  the material derivatives in  \eqref{eqinv:tensors+}  reduce to just time derivatives and the equation becomes a nonlinear damped wave equation. Morally speaking it will be this damping that is responsible for the global existence even in the case when the flow is present. Thus we have:  
\bigskip
\begin{theorem}\label{intro:thm2}
{\bf [Global existence and uniqueness for small initial data]}
	 
	  Consider the system \eqref{eqinv:momentum+},\eqref{vdivfree},\eqref{eqinv:tensors+}.We assume that $J<J_0$, $\mu_1>\bar \mu_1$, $\beta_4>\tilde C_d$  (where    $J_0:=J_0(\bar\mu_1,a,b,c)$, $\bar\mu_1=\bar\mu_1(a,b,c)$,  and $\tilde C_d=\tilde C_d(\tilde\mu_2,\beta_5,\beta_6,\mu_2)>0$ are explicitly computable coefficients). Furthermore  we assume the positivity of a number of coefficients: $\beta_1,\mu_1> 0$, $a>0$ and $J,L>0$.

	Let  $(v_0,\,Q_0): \RR^d \to \RR^d\times S_0^{(d)}$ be in  $H^s(\RR^d)\times H^{s+1}(\RR^d)$ with $s>\frac{d}{2}$ and $d=2$ or $d=3$. 
 Then there exists $\varepsilon_0>0$, depending on $s$ and $d$ such that  if 
		\begin{equation*}
		\eta_0:= \| v_0 \|_{H^s} + \| Q_0 \|_{H^{s+1}}+\|\dot Q_0\|_{H^s}<\varepsilon_0
	\end{equation*}
	then there exists a unique strong solution $(v,\,Q)$ of 
	\eqref{eqabst:momentum}-\eqref{eqabst:tensors}, which is global in time. Moreover
	there exists a positive constant $C$ (independent of the solution) such that
	\begin{align*}
		\| 			v \|_{L^\infty(\RR_+; H^s(\RR^d))		} + 
		\| \nabla 	v \|_{L^2(\RR_+;H^{s}(\RR^d))	}		 &+
		\| 			Q \|_{L^\infty(\RR_+; H^{s+1}(\RR^d))	} +
		\| 			Q \|_{L^2(\RR_+; H^{s+1}(\RR^d))			} +\\&+
		\|	\dot 	Q \|_{L^\infty(\RR_+; H^{s}(\RR^d))		} +	
		\|	\dot 	Q \|_{L^2(\RR_+; H^{s}(\RR^d))			} \leq C\eta_0.
	\end{align*}
\end{theorem}

One important assumption in the above theorem is that the coefficient $a>0$. This captures a regime  of physical interest but unfortunately  not the most interesting physical regime (which would be for $a\le 0$ the ``deep nematic" regime, see \cite{NewtonMottram}). Technically this assumption that $a>0$ provides a sort of additional {\it ``damping in time"} that was used previously in  related settings  in \cite{Xiaoglobal}, respectively \cite{KWZ}.

\bigskip
The  difficulties associated with treating the system \eqref{eqinv:momentum+},\eqref{vdivfree},\eqref{eqinv:tensors+} are generated, as usually (in this type of non-Newtonian fluid) by the forcing term in the Navier-Stokes part. One can essentially think of the system as a highly non-trivial perturbation of a Navier-Stokes system. 

\smallskip
However the specific difficulty in our system is that the forcing term involves $Q$ whose equation is now of hyperbolic nature. This is due to the inertial term that contains {\it a second order material derivative}. We are not aware of any systematic treatment of such a term in other contexts, but it turns out that the most delicate part of the whole proof is related to its treatment. Indeed, one should start by noticing that this is very far from the case when $v=0$ (when it is just $\partial^2_t Q$) because it is a highly nonlinear operator, for instance for $v$ and $Q$ smooth we explicitly have:

\be
\ddot Q=\partial^2_t Q+2(v\cdot\nabla)\partial_t Q+(v_t\cdot\nabla) Q+\left((v\cdot\nabla)v\cdot \nabla\right)Q+v\nabla^2 Q v 
\ee involving an {\it ``expensive derivative of v"} i.e. $\partial_t v$ and the term $v\nabla^2 Q v$ that competes in a sense with the regularizing Laplacian $L\Delta Q$ on the right hand side of the $Q$-equation.

Our main trick in dealing with the double material derivative has been to stay as close as possible to the standard cancellation appearing in the context of convective derivatives, which can be formally written as $\int_{\R^d}\left( f\cdot\nabla \right) v f\,dx=0$ for $f$ decaying sufficiently fast at infinity. However, in order to implement this we have to use a higher-order commutator estimate that appears in \cite{commutator}, an estimate which is at the level of homogeneous Sobolev spaces, $\dot H^s$. This is very convenient for our purposes because  the $H^s(\RR^d)$ norm does not allow the cancellation of the worst terms, as in obtaining  \eqref{energydecay} in the $L^2(\RR^d)$-setting. This difficulty is partially dealt with by reformulating the inner product of $H^s(\RR^d)$ into\vspace{-0.3cm}
\begin{equation*}
	\langle \omega_1, \omega_2 \rangle_{L^2(\RR^2)} + \langle \omega_1, \omega_2 \rangle_{\dot H^s(\RR^2)}
	=\int_{\RR^d_\xi}\big( 1+ |\xi|^{2s}\big)\hat \omega_1(\xi)\hat \omega_2 (\xi)\dd \xi,\vspace{-0.3cm}
\end{equation*}
where $\dot H^s(\RR^d)$ stands for the homogeneous Sobolev space with index $s$. It is straightforward that this inner product generates the same topology in $H^s(\RR^d)$ with respect to the common one.

%, given by\vspace{-0.2cm}
%\begin{equation*}
%	\langle \omega_1, \omega_2 \rangle_{H^s(\RR^2)} =\int_{\RR^2_\xi}\big( 1+ |\xi|\big)^{2s}\hat \omega_1(\xi)\hat \omega_2 (\xi)\dd \xi.
%	\vspace{-0.2cm}
%\end{equation*}
%Making use of this approach, one essentially reduces the control of the worst terms only in $\dot H^s(\RR^2)$, where useful commutator estimates hold. 

Our main work on proving the existence of classical solutions is to obtain an uniform estimate for our approximate solutions, that is to close an estimate of the type:
\begin{equation}\label{intro:estimate_existence}
	\Phi'(t) +  \Psi (t) \leq  C \Phi(t)\Psi(t),
\end{equation}
where $C$ is a suitable positive constant, $\Phi$ is controlling the  $H^s$-norms in space for our solution and $\Psi$ is an  integrable-in-time quantity involving $H^s$-norms. Then, a rather standard argument (see for instance in the Appendix, Lemma~\ref{appx:lemma1}) allows to propagate the smallness condition on the initial data (i.e. on $\Phi(0)$). This leads the right-hand side of the above equation to be absorbed by the left-hand side, achieving the cited uniform estimates. Finally we construct our classical solution, through a compactness method.

\smallskip
The uniqueness of our solutions is proven by  evaluating the difference between two solutions at a regularity level $s=0$, i.e. in $L^2(\RR^2)$. Our work is mainly to obtain an estimate that leads to the Gronwall lemma. Here the main difficulties are handled taking into account a specific feature of the coupling system related to the difference of the two solutions. This feature allows the cancellation of the worst term when considering certain physically meaningful combination of terms.

It is perhaps interesting to remark that in Theorem \ref{intro:thm2} we do not need consider a positive constant $c$ in the bulk free energy density $\psi_B(Q)$. Usually, this is a necessary condition in order to have $\psi_B(Q)$ bounded from below the space of $Q$-tensors\footnote{ also $c>0$ is necessary for well-posedness for large data, as can be seen by looking at the $Q$ equation where we take $u=0$ and $J=0$}. However we do not need this restriction on $c$ mainly because we are assuming a smallness condition on the initial data, smallness which will be propagated by the equation. 

%This smallness property is preserved by our solutions, so that one can heuristically think that $\psi_B(\cdot)$ only acts on  $Q$-tensors whose components belong to a bounded domain in  $S_0^{(d)}$. Thus, in this functional framework, $\psi_B$ is still bounded from below. %Moreover, Theorem \ref{intro:thm1} requires the constant $a$ to be positive. A first reason for this restriction releases again in the smallness condition, since $a$ is the constant related to the lower power of $Q$, which in this contest has the same behaviour of $\psi_B(Q)$. Nevertheless the main reason for the positivity of $a$ concerns a technical part on proving \eqref{intro:estimate_existence}, that is the time boundary and $L^2$-integrability for the $H^s$-norm of $Q$.

\bigskip Finally, one last issue of interest to us are the so-called {\it ``twist waves"}. These are  solutions of the coupled system, for which the flow $v$ is identically zero. The existence of such solutions for the Ericksen-Leslie system was first postulated by J.L. Ericksen in \cite{Ericksentwist}, who named them ``twist waves" and provided one explicit example. 

We note that if $v=0$, the $Q$-tensor evolution \eqref{eqinv:tensors+} reduces to a nonlinear wave system:

\be\label{eq:Qwave0}
JQ_{tt}+\mu_1 Q_t=L\Delta Q-aQ+b(Q^2-\frac{|Q|^2}{3}I_3)-cQ|Q|^2
\ee

Nevertheless, there is a part of the momentum equation \eqref{eqinv:momentum+} that survives as an additional constraint on $Q$, namely:

\be\label{eq:Qtwistconst0}
\nabla p=\nabla\cdot \bigg(-\nabla Q\otimes \nabla Q+\frac{\mu_2}{2}Q_t+\mu_1\, \big[Q,Q_t\big]\bigg)
\ee 
Clearly, because of this additional constraint, only very special types of initial data for \eqref{eq:Qwave0} will generate solutions that respect the constraint \eqref{eq:Qtwistconst0}. One example is obtained by taking in $\R^d$ with $d=2$ or $d=3$ the ansatz: 

$$T(t,x):=f(t,|x|)\bar H(x)$$ with $f:\R_+\times \R\to \R$ a function to be determined and $\bar H$ the ``hedgehog" function (see \cite{meltinghedgehog} for details about its physical significance) :

$$\bar H_{ij}(x):=\frac{x_ix_j}{|x|^2}-\frac{\delta_{ij}}{d}, i,j=1,\dots,d$$ 

Then the \eqref{eq:Qwave0} reduces to:

\be\label{eq:s0}
Jf_{tt}+\mu_1 f_t=L\left(f_{rr}+f_r\frac{d-1}{r}-\frac{2d}{r^2}f\right)-af+\frac{b(d-2)}{d}f^2-\frac{c(d-1)}{d}f^3
\ee

We can then show:

\begin{proposition}\label{prop:meltingwave} 
Let $f_0:\R\to\R$ be a smooth  function such that 

$$f(0)=f_r(0)=0.$$

 Let $d=2$ and assume that  in \eqref{eq:s0} we have $J,L,c,\mu_1>0$. Then there exists a function $f:\R_+\times \R \to\R$, that is $C^2$ on $(0,\infty)\times (0,\infty)$, and such that both $f$ and $f_t$ are smooth functions of $r$ on $[0,\infty)$,  solution of the equation \eqref{eq:s0} witht $f(0,r)=f_0(r)$ so that the function $T(t,x)=f(t,|x|)\bar H(x)$ is a smooth twist wave in $\R^d$, i.e. a classical solution of \eqref{eq:Qwave0} satisfying the constraints  \eqref{eq:Qtwistconst0}.

\end{proposition}

\begin{remark}
A similar statement can be made for $d=3$ but for technical reasons we are unable to show the global existence of the twist-wave in this case, but only the local existence in time.
\end{remark}
}
{\noindent {\textbf{Organization of the paper}}

\smallskip
This paper is organised as follows: in the next section we prove Theorem \ref{intro:thm1} concerning the energy law, in Section~\ref{sec2_1/2} we present  apriori estimates for higher norms and prove the global existence result, namely  Theorem \ref{intro:thm2}. Finally in  Section~\ref{sec4} we construct a specific twist-wave solution, providing the proof of Proposition~\ref{prop:meltingwave} . A number of natural open problems are proposed and discussed in Section~\ref{sec:openprob}.
%-----------------------------------------------------------------------------------------------------------------------------------------------
%-----------------------------------------------------------------------------------------------------------------------------------------------
%-----------------------------------------------------------------------------------------------------------------------------------------------

\bigskip {\bf\noindent Notations and conventions} 

We denote by $\dot f$ the material derivative $\dot f=\delt h+v\cdot f$, where the fluid velocity $v$ is understood from the context. We also use the Einstein summation convention, that is summation over repeated indices. We denote weighted spaces, by specifing the weighted measure, for instance $L^2(\R,r^2dr)=\{f:\R\to\R, \int_{\R} f^2r^2\,dr<\infty\}$. We also use the notation $\|(f,g)\|_X=\|f\|_X+\|g\|_X$ for any elements $f,g\in X$ with $X$ a suitable normed space with norm $\|\cdot\|_X$. We denote by $\mathcal{F}$ and $\mathcal{F}^{-1}$ respectively the Fourier transform and its inverse.

We use $\eps_{ijk}$, the Levi-Civita symbol, with indices $i,j,k$ from $1$ to $d$, and the comma in the subscript denotes derivative with respect to particular spatial coordinate. If $M(x)$ is a $d\times d$-matrix, then $\nabla \cdot M$ stands for the vector field $(M_{ij,j})_{i=1\dots,d}$. The $|M|$ denotes the Frobenius norm of the matrix, i.e. $|M|=\sqrt{MM^t}$.

The $[\cdot,\cdot]$ stands for the usual commutator bracket $[A,B]:=AB-BA$, for any $d\times d$-matrices $A$ and $B$. We denote $=\tr(AB)$ with $A:B$. The product $\nabla A\otimes \nabla B$ is a matrix with $ij$ component $\left(\partial_i A: \partial_j B\right)$.

\section{Energy  law and apriori bounds}\label{sec3}

%\begin{remark}
% Let us Remark that the matrix $[Q,\, [\Omega,\,Q]]$ on the right hand side of the momentum equation is skew-adjoint. Hence, multiplying by 
% $u$ and integrating such term in $\R^d$ we get 
% \begin{align*}
% 	+\mu_1 \int_{\R^d}\nabla \cdot  [Q,\, [\Omega,\, Q]]\cdot u = +\mu_1 \int_{\R^d} \partial_j[Q,\, [\Omega,\, Q]]_{ji}u_i =
% 	-\mu_1 \int_{\R^d} [Q,\, [\Omega,\, Q]]_{ji}u_{i,j}=\\= -\mu_1\int \trc\{\,^t [Q,\, [\Omega,\, Q]]\nabla u\}=
% 	-\mu_1\int \trc\{ [Q,\, [\Omega,\, Q]]\nabla u\}.
% \end{align*}
%\end{remark}
%\begin{lemma} Let $(v,Q)$ be a smooth solution of \eqref{eqinv:momentum+},\eqref{vdivfree}, \eqref{eqinv:tensors+} in the whole space $\R^d$,  and decaying sufficiently fast at infinity. Assume that the Newtonian viscosity $\beta_4$ is positive and large enough compared to the other viscosities $\beta_1,\beta_5,\beta_6,\mu_1$ and $\mu_2$ with $\mu_1>0$. Furthermore assume that $Q$ is in the physical range initially (i.e. $\int_{\R^d}\psi_B(Q_0(x))\,dx<\infty$.

%Then we have the a priori bounds:

%$$\|v\|_{L^\infty(0,T;L^2)\cap L^2(0,T;H^1)}, |\dot Q|_{L^\infty(0,T; L^2)}, |\nabla Q|_{L^\infty(0,T;L^2)},|Q|_{L^\infty((0,T)\times \R^d)}<\infty$$

%\end{lemma}

\smallskip
{\bf Proof of Theorem~\ref{intro:thm1}:} We multiply the equation \eqref{eqinv:momentum+} by $v_i$ integrate over the space and by parts\footnote{without boundary terms, thanks to our assumptions} and use \eqref{vdivfree} to cancel some terms.  To the result obtained we add  equation \eqref{eqinv:tensors+} multiplied by $\dot Q_{ij}$, integrated over the space and by parts \footnote{again without boundary terms thanks to our assumptions} to obtain\footnote{ where $\sigma'$ was defined in \eqref{defsigmaprime}, $h'$ in \eqref{def:hprime} and $\psi_B$ in \eqref{def:psib}. The operator $\mathcal{L}$ denotes the projection onto trace-free matrices.}:

\bea
\frac{d}{dt}\int_{\R^d} \frac{1}{2}\big(|v|^2+J|\dot Q|^2\big)\,dx&=\int_{\R^d} LQ_{kl,j}Q_{kl,i}v_{i,j}+\left(L\Delta Q_{ij}-\mcL\frac{\partial \psi_B}{\partial Q}_{ij}
\right)\left(\delt Q_{ij}+v\cdot\nabla Q_{ij}\right)\,dx\non\\
&-\int_{\R^d} \sigma'_{ij}v_{i,j}\,dx+\int_{\R^d} h'_{ij}\dot Q_{ij}\,dx\non\\
&=L\underbrace{\int_{\R^d} Q_{kl,j}Q_{kl,i}v_{i,j}+\Delta Q_{ij}v_kQ_{ij,k}\,dx}_{:=\mcI_1}-L\underbrace{\int_{\R^d}\delt Q_{ij,k}Q_{ij,k}}_{:=\mcI_2}\non\\
&-\underbrace{\int_{\R^d} \delt Q_{ij} \frac{\partial \psi_B}{\partial Q}_{ij}\,dx}_{:=\mcI_3}-\underbrace{\int_{\R^d} v\cdot\nabla Q_{ij}\frac{\partial \psi_B}{\partial Q}_{ij}}_{:=\mcI_4}\non\\
&- \underbrace{\beta_4 \int_{\R^d}A_{ij}v_{i,j}\,dx}_{:=\mcI_5}-\beta_1\int_{\R^d} Q_{ij}Q_{lk}A_{lk}v_{i,j}\,dx\non\\
&-\beta_5 \int_{\R^d}A_{il}Q_{lj}v_{i,j}\,dx-\beta_6\int_{\R^d} Q_{il}A_{lj}v_{i,j}\,dx\non\\
&-\frac{1}{2}\mu_2\int_{\R^d}\left(\dot Q_{ij}-\Omega_{ik}Q_{kj}+Q_{ik}\Omega_{kj}\right)v_{i,j}\,dx\non\\
&-\mu_1 \int_{\R^d}\Big( Q_{il} \dot Q_{lj} - \dot Q_{il} Q_{lj}\Big)v_{i,j} dx\non\\
&+\mu_1 \int_{\R^d} \left(Q_{il}[\Omega,\,Q]_{lj}-[\Omega,\,Q]_{il}Q_{lj}\right)v_{i,j}\,dx\non\\
&+\frac{\tilde\mu_2}{2}\int_{\R^d} A_{ij}\dot Q_{ij}-\mu_1 \int_{\R^d}\left(\dot Q_{ij}-\Omega_{ik}Q_{kj}+Q_{ik}\Omega_{kj} \right) \dot Q_{ij}\,dx
\eea 

Noting that thanks to  \eqref{vdivfree} we have $\mcI_1=\mcI_4=0$ and moving $\mcI_2,\mcI_3,\mcI_5$ on the left hand size, the last relation becomes 

\begin{equation}
\begin{aligned}
	\frac{d}{dt}\int_{\R^d} &\frac{1}{2}
		\big(
			|v|^2+J|\dot Q|^2+L|\nabla Q|^2
		\big)+\psi_B(Q)\,dx	
	+\frac{\beta_4}{2}	\int_{\R^d}	|\nabla v|^2\,dx				+\mu_1	\int_{\R^d} |\dot Q|^2\,dx															\\
	&=	-\beta_1\int_{\R^d} Q_{ij}Q_{lk}A_{lk}v_{i,j}\,dx			-\beta_5 \int_{\R^d}A_{il}Q_{lj}v_{i,j}\,dx
		-\beta_6\int_{\R^d} Q_{il}A_{lj}v_{i,j}\,dx																												\\
	&\quad-\frac{1}{2}\mu_2\int_{\R^d}\left(\dot Q_{ij}-\Omega_{ik}Q_{kj}+Q_{ik}\Omega_{kj}\right)v_{i,j}
	+\frac{\tilde\mu_2}{2}\int_{\R^d} A_{ij}\dot Q_{ij}\\
	&\quad+\mu_1 \int_{\R^d}\Big( \dot Q_{il}Q_{lj}  - Q_{il}\dot Q_{lj} \Big)v_{i,j} dx
	+\mu_1 \int_{\R^d}\left(\Omega_{ik}Q_{kj}-Q_{ik}\Omega_{kj} \right) \dot Q_{ij}\,dx\\
	&\quad-\mu_1 \int_{\R^d} \left([\Omega,\,Q]_{il}Q_{lj}-Q_{il}[\Omega,\,Q]_{lj}\right)v_{i,j}\,dx\,dx
\end{aligned}
\end{equation}
Now we analyse each term on the right-hand side of the equality, and we will repeatedly use that $v_{i,j} = A_{ij} + \Omega_{ij}$ and moreover that $\trc\{ BC \}=B:C$ is null for any $B$ symmetric and $C$ skew-adjoint.  We begin with
\begin{align*}
	-\beta_1 \int_{\R^d} Q_{ij}Q_{lk}A_{lk}v_{i,j} 	&= 	-\beta_1 \int_{\R^d} Q_{ij}Q_{lk}A_{lk}A_{ij} 	- \beta_1 \int_{\R^d} Q_{ij}Q_{lk}A_{lk}\Omega_{ij}\\
													&=	-\beta_1 \int_{\R^d} \left(Q:A\right)^2  			- \beta_1 \int_{\R^d} (Q:\Omega)(Q:A)\\
													&=	-\beta_1 \int_{\R^d} \left(Q:A\right)^2,
\end{align*}
observing that $Q:\Omega=0$. 
Furthermore we have:
\begin{align*}
	&-\mu_1 \int_{\R^d}\Big( Q_{il} \dot Q_{lj} - \dot Q_{il} Q_{lj}\Big)v_{i,j} dx
	+\mu_1 \int_{\R^d}\left(\Omega_{ik}Q_{kj}-Q_{ik}\Omega_{kj} \right) \dot Q_{ij}\,dx=\\
	&=\underbrace{\mu_1\int_{\R^d} \trc\{(Q\dot Q -\dot Q  Q)A\}}_{=0}+\mu_1\int_{\R^d} \trc\{(Q\dot Q - \dot Q  Q)\Omega\}
	+\mu_1\int_{\R^d}\trc\{(\Omega Q -Q \Omega)\dot Q\}=\\
	&=2\mu_1\int_{\R^d}\trc\{[\Omega,\,Q]\dot Q\}.
\end{align*}Finally
\begin{align*}
	&-\mu_1 \int_{\R^d} \left([\Omega,\,Q]_{il}Q_{lj}-Q_{il}[\Omega,\,Q]_{lj}\right)v_{i,j} =
	\mu_1 \int_{\R^d}\trc\bigg\{\big([\Omega,\,Q]Q- Q[\Omega,\,Q]\big)\Omega\bigg\}=\\
	&-\mu_1 \int_{\R^d}\trc\big\{(\Omega Q -Q\Omega) [\Omega,\,Q]\big\}=
	-\mu_1 \int_{\R^d}|[\Omega,\,Q]|^2.
\end{align*}

Now we deal with
\begin{align*}
	-\beta_5 &\int_{\R^d}A_{il}Q_{lj}v_{i,j}\,dx	- \beta_6\int_{\R^d} Q_{il}A_{lj}v_{i,j}\,dx	
																			+\frac{\mu_2}{2}\int_{\R^d}\big(\Omega_{ik}Q_{kj}-Q_{ik}\Omega_{kj}\big)v_{i,j}\,dx = 					\\
	&= -\beta_5 \int_{\R^d } \trc\{ QA\nabla v \}\, -\, \beta_6 \int_{\R^d } \trc\{ AQ\nabla v \} \,+\, \frac{\mu_2}{2}\int_{\R^d}\trc\{ [\Omega,\,Q]\,\nabla v \}					\\
	 &= -\beta_5 \int_{\R^d } \trc\{ (QA+AQ)\nabla v \}\,- \,(\beta_6 - \beta_5)\int_{\R^d } \trc\{ AQ \nabla v \}	\,+\,\frac{\mu_2}{2}\int_{\R^d}\trc\{ A\,[\Omega,\,Q] \}			\\
	&= -\beta_5 \int_{\R^d } \trc\{ (QA+AQ)A \}\, -\, (\beta_6-\beta_5)\int_{\R^d } \trc\{ AQ (A+\Omega) \} \,+\,\frac{\mu_2}{2} \int_{\R^d}\trc\{ A\,[\Omega,\,Q]\}				\\
	&= -2\beta_5 \int_{\R^d } \trc\{ AQA\}-(\beta_6-\beta_5)\int_{\R^d } \trc\{ AQ A \}\,+\,
	        \left(\,
	            \frac{\beta_6-\beta_5}{2}\,+\,\frac{\mu_2}{2} \,
	        \right)
	        \int_{\R^d} \trc\{A[\Omega,\,Q]\}\\
	&= \,-\, (\beta_5+\beta_6)\int_{\R^d}\trc\{ AQA\}\,+\,\mu_2\int_{\R^d} \trc\{A[\Omega,\,Q]\}.
\end{align*}

We arrange  the remaining terms related to $\mu_2,\tilde\mu_2$ as
\begin{align*}
	-\frac{1}{2}\mu_2\int_{\R^d}\dot Q_{ij}v_{i,j}
	+\frac{\tilde\mu_2}{2}\int_{\R^d} A_{ij}\dot Q_{ij}
	&= -\frac{\mu_2}{2}\int_{\R^d} \dot Q:\nabla v 
	+\frac{\tilde\mu_2}{2}\int_{\R^d}A:\dot Q\\
	&= -\frac{\mu_2}{2}\int_{\R^d} \dot Q: A  -\underbrace{\frac{\mu_2}{2}\int_{\R^d} \dot Q:\Omega}_{=0}
	+\frac{\tilde\mu_2}{2}\int_{\R^d}A:\dot Q\\	
	&=\left\{\begin{array}{ll} 0 &\textrm{ if }\tilde\mu_2=\mu_2\\
	-\mu_2\int_{\R^d} \dot Q: A  &\textrm{ if }\tilde\mu_2=-\mu_2
	\end{array}\right.
%	\lesssim \int_{\R^d} |\dot Q|^2+C_{\Gamma, L}\int_{\R^d} |\nabla v|^2.
\end{align*}

Summarizing all the previous estimates, we get:

\begin{equation}\label{1st-energy-estimate}
\begin{aligned}
	\frac{d}{dt}&\int_{\R^d} \frac{1}{2}
		\big(
			|v|^2+J|\dot Q|^2+L|\nabla Q|^2
		\big)+\psi_B(Q)\,\dd x	
	+\frac{\beta_4}{2}	\int_{\R^d}	|\nabla v|^2\,\dd x	 +\\&
	+\beta_1 \int_{\R^d} (Q:A)^2\dd x
	+\mu_1\int_{\R^d}| \dot Q - [\Omega, Q]|^2\dd x\\&
	=\,\tilde \mu_2\int_{\R^d}\dot Q: A\,+\,
	\,\mu_2\int_{\R^d}\trc\{A[\Omega,\,Q]\},\\
\end{aligned}
\end{equation}
 where we have used the assumption $\beta_5\,+\,\beta_6\,=\,0$, in \eqref{parodi-corot+}. 
We recall now the condition we impose on $\mu_2$ namely \eqref{mu2cond}, so that if $\tilde \mu_2=\mu_2=0$ then the right hand side of the above equation is null. Otherwise, 
if $\tilde \mu_2=-\mu_2$, recalling that  $\mcN=\dot Q-[\Omega,Q]$ we obtain
\begin{equation*}
\,\tilde \mu_2\int_{\R^d}\dot Q: A\,+\,
	\,\mu_2\int_{\R^d}\trc\{A[\Omega,\,Q]\}\,=\,\tilde\mu_2\int_{\R^d}A:\mcN.
\end{equation*}

In both cases,  we note out of the above that there exists a constant $C_d$ depending on $\tilde\mu_2,\beta_5,\beta_6,\mu_2$ such that if $\beta_4>C_d$ then the total  energy decays, i.e.
\be\label{eq:estineqbd}
\frac{d}{dt}\int_{\R^d} \frac{1}{2}
		\big(
			|v|^2+J|\dot Q|^2+L|\nabla Q|^2
		\big)+\psi_B(Q)\,\dd x\le 0
\ee

 If $d=2$ then the term $\textrm{tr}(Q^3)$ vanishes and if furthermore  $a\ge 0$ then $\psi_B(Q)\ge 0$ and the previous estimate provides the claimed apriori bounds.  However, in general  the estimate \eqref {eq:estineqbd} does not suffice for obtaining apriori bounds on the norms of the solutions, because $\psi_B(Q)$ can be negative.In order to deal with this we need to obtain apriori control on suitable $L^p$ norms of $Q$.

We multiply  \eqref{eqinv:tensors+} by $Q$, take the trace and integrate over space, obtaining: 
\begin{align*}
	J		\int_{\R^d}		\ddot Q_{\alpha\beta} Q_{\alpha\beta} \dd x+ 
	\mu_1	\int_{\R^d}		\dot  Q_{\alpha\beta} Q_{\alpha\beta} \dd x-
	\underbrace{
	\mu_1	\int_{\R^d}		\big( \Omega_{\alpha\gamma} Q_{\gamma\beta}-Q_{\alpha\gamma}\Omega_{\gamma\beta}\big)Q_{\alpha\beta}\dd x}_{=0} -\\- 
			\int_{\R^d}		L\Delta Q_{\alpha\beta} Q_{\alpha\beta}\dd x= 
%-------------------------------------------------------------------------------------------------------------------------------			 	%-------------------------------------------------------------------------------------------------------------------------------
	\underbrace{
			\int_{\R^d} 	\big( -a Q_{\alpha\beta}Q_{\alpha\beta} + b Q_{\alpha\gamma}Q_{\gamma\beta}Q_{\beta\alpha}  				-
			                                                      c \big( Q_{\alpha\beta}Q_{\alpha\beta} \big)^2 \big)\dd x}_{:=P(Q)} 	+\\+
	\frac{\tilde\mu_2}{2}	\int_{\R^d}		A_{\alpha\beta}Q_{\alpha\beta}\dd x.	
\end{align*}
Now, let us remark that
\begin{align*}
	J							\int_{\R^d}		\ddot Q_{\alpha\beta} Q_{\alpha\beta} \dd x &=
	J							\int_{\R^d}		\partial_t 					\dot Q_{\alpha\beta} Q_{\alpha\beta} + 
												v_\gamma	 	\dot Q_{\alpha\beta,\gamma} Q_{\alpha\beta}\dd x \\&=
	J							\int_{\R^d}		\partial_t					\big( 	\dot Q_{\alpha\beta} Q_{\alpha\beta} 	\big) + 
												v_\gamma\big(	\dot Q_{\alpha\beta} Q_{\alpha\beta} 	\big)_{,\gamma}\dd x -			
	J							\int_{\R^d}		\dot Q_{\alpha\beta}\dot Q_{\alpha\beta}\dd x\\&=
	J		 \frac{\dd}{\dd t}	\int_{\R^d}		\frac 12 |\dot Q + Q|^2 - \frac 12 |\dot Q|^2 -\frac 12 |Q|^2\dd x - 
	J							\int_{\R^d}		|\dot Q|^2\dd x		 			
\end{align*}
and moreover
\begin{align*}
	\mu_1									\int_{\R^d}		\dot  Q_{\alpha\beta} Q_{\alpha\beta}\dd x -
											L\int_{\R^d}		\Delta Q_{\alpha\beta} Q_{\alpha\beta}\dd x &= 
	\frac{\mu_1}{2}	\frac{\dd}{\dd t}		\int_{\R^d}		|Q|^2\dd x + \\ &+
	\underbrace{
	\mu_1									\int_{\R^d}		v_{\gamma}  Q_{\alpha\beta} Q_{\alpha\beta,\gamma}\dd x}_{=0}  +
											L\int_{\R^d}		|\nabla Q|^2\dd x.
\end{align*}
Thus, summarizing, it turns out that
\begin{equation}\label{2nd-energy-estimate-part1}
\begin{aligned}
	\frac 12 \frac{\dd}{\dd t}	\int_{\R^d}		J |\dot Q + Q|^2 - J |\dot Q|^2 +(\mu_1-J) |Q|^2\dd x - 
	J							\int_{\R^d}		|\dot Q|^2\dd x	 +	 			
	L\int_{\R^d}		|\nabla Q|^2\dd x =\\ = 
	%-------------------------------------------------------------------------------------------------------------------------------			 	%-------------------------------------------------------------------------------------------------------------------------------
	P(Q) +
	\frac{\tilde\mu_2}{2}	\int_{\R^d}	\trc\{	AQ\}\dd x.	
\end{aligned}
\end{equation}
{ We will use this estimate together with \eqref{1st-energy-estimate} to obtain an estimate of the $L^2$ norm of $Q$ which then will allow to obtain out of  \eqref{1st-energy-estimate} the desired apriori estimates on $Q,\dot Q$ and $v$.

We note that if $Q$ has eigenvalues $\lambda,\mu,-\lambda-\mu$ (as it is traceless) we have $|Q|^2=2(\lambda^2+\mu^2+\lambda\mu)$ and $\tr(Q^3)=-3\lambda\mu(\lambda+\mu)$ thus for any $\delta>0$ we have $|\tr(Q^3)|\le \frac{3\delta}{8}|Q|^4+\frac{3}{2\delta}|Q|^2$. Furthermore $\tr(Q^3)=0$ if $d=2$ since $Q$ is a two-by-two traceless symmetric matrix. If $d=3$ we claim that there exists $\bar\mu_1>0$ depending on $a,b$ and $c>0$ such that 

\be\label{est:psibq2}
\bar\mu_1|Q|^2+4\psi_B(Q)>\eps |Q|^2
\ee for some $\eps>0$.

Indeed, we have $\bar\mu_1|Q|^2+4\psi_B(Q)-\eps|Q|^2=(\bar \mu_1+2a-\eps)|Q|^2-\frac{2b}{3}\textrm{tr}(Q^3)+\frac{c}{2}|Q|^4\ge
(\bar \mu_1+2a-\eps)|Q|^2-\frac{2|b|}{3}(\frac{3\delta}{8}|Q|^4+\frac{3}{2\delta}|Q|^2)+\frac{c}{2}|Q|^4$. Thus taking $\frac{3\delta}{8}\frac{2|b|}{3}=\frac{c}{4}$ and letting $\bar\mu_1$ be large enough we obtain the claimed relation \eqref{est:psibq2}.
 Then, assuming that $J<\mu_1$ and adding \eqref{2nd-energy-estimate-part1}  to twice times \eqref{1st-energy-estimate}
 we obtain
\begin{equation}\label{1st-energy-estimate-part1}
\begin{aligned}
	\frac{d}{dt}&\int_{\R^d} |v|^2+\frac{J}{2}\left(|\dot Q + Q|^2+|\dot Q|^2\right)+L|\nabla Q|^2+\frac 12(\mu_1-J) |Q|^2+2\psi_B(Q)
		\,\dd x	+\\&+L\int_{\R^d}|\nabla Q|^2 +
		2\beta_1 \int_{\R^d} \trc\{ QA \}^2+2\beta_4	\int_{\R^d}	|\nabla v|^2\dd x	+2\mu_1	\int_{\R^d} |\dot Q-[\Omega,Q]|^2\dd x\\&=P(Q)+ J\int_{\R^d} |\dot Q|^2\dd x+\frac{\tilde\mu_2}{2}\int_{\R^d}\trc\{AQ\}\dd x+2\,\tilde \mu_2\int_{\R^d}A: \mcN.
\end{aligned}
\end{equation}

Thus for  $\mu_1>\bar\mu_1$,  $J$ small enough and $\beta_4$ large enough we have a Gronwall-type inequality for the $L^2$ norm of $Q$ which then can be combined with \eqref{1st-energy-estimate} to obtain the apriori bounds \eqref{aprioribounds}. $\Box$

\section{Global strong solutions}\label{sec2_1/2}

\subsection{A priori high-norm estimates}
In this subsection we provide the apriori estimates that exhibit in a relatively simple  setting the higher-order  cancellations and estimates that will allow us to prove afterwards the existence of strong solutions through a suitable approximation scheme, in the next subsection.
 
We consider the inhomogeneous Sobolev space $H^s$ with $s>\frac{d}{2}$, equipped with inner product
\begin{equation*}
	\langle u,\,v\rangle_{H^s} = \langle u,\,v \rangle_{L^2} + \langle u,\,v\rangle_{\dot H^s}.
\end{equation*}

where 

$$
\langle u,v\rangle_{\dot H^s}:=\langle (\sqrt{-\Delta})^s u, (\sqrt{-\Delta})^s v\rangle_{L^2}
$$ with $({\sqrt{-\Delta})^s u}(\xi):=\mathcal{F}^{-1}\left(|\xi|^s \mathcal{F} u(\xi)\right)$.
 
We recall that for $s>\frac{d}{2}$ we have $H^s(\R^d)\hookrightarrow L^\infty(\R^d)$ and, most importantly for our purposes, it is an algebra, with 

$$\|uv\|_{H^s}\le \|u\|_{H^s}\|v\|_{H^s}.$$

\bigskip We assume that the solutions are suitably smooth and decaying sufficiently fast at infinity to be able to integrate by parts without boundary terms whenever necessary. Taking the $H^s$ product between \eqref{eqinv:momentum+} and $v$, we get
\begin{equation}\label{prop:est1}
\begin{aligned}
	\frac 12 \frac{\dd }{\dd t} 	\| v\|_{H^s}^2	&+ \beta_4 \|\nabla v\|_{H^s}^2 = 
	-						\langle		v\cdot \nabla v		,\, 		   v	\rangle_{H^s}
	+ 						L\langle	\nabla Q\odot \nabla Q	,\, 	\nabla v	\rangle_{H^s}-\\&
	- \beta_1 				\langle 	\tr\{AQ\}	Q		,\,		\nabla v	\rangle_{H^s} 
	- \beta_5				\langle 				AQ		,\,		\nabla v	\rangle_{H^s} 
	- \beta_6				\langle					QA		,\,		\nabla v	\rangle_{H^s}+\\&
	- \frac{\mu_2}{2}		\langle		\dot Q-[\Omega,\,Q]	,\,		\nabla v	\rangle_{H^s}
	-\mu_1					\langle		[Q,\,\dot Q]		,\,		\nabla v\rangle_{H^s}
	+\mu_1					\langle		[Q,\,[\Omega,\,Q]]	,\,		\nabla v	\rangle_{H^s}
\end{aligned}
\end{equation} 
Now, let us observe that
\begin{equation*}
	\langle		v\cdot \nabla v		,\, 		 v	\rangle_{H^s} =
	\underbrace{\langle v\cdot \nabla v,\,v\rangle_{L^2_x}}_{=0} + 
	\langle v\cdot \nabla v,\, v\rangle_{\dot H^s}.
\end{equation*}
%------------------------------------------------------------------------------------------------------------------------
	
	Since $s>d/2$, then $H^{m}(\RR^d)$  is continuously embedded in $L^\infty(\RR^d)$,  for $m$  a natural number in $[s, 1+s)$. Then, by the classical Gagliardo-Niremberg 
	inequality we have
	\begin{align*}
		\| v \|_{L^\infty(\RR^d)}\lesssim \| v\|_{L^2(\RR^d)}^\theta \|  v \|_{\dot H^m(\RR^d)}^{1-\theta}
		&= \| v\|_{L^2(\RR^d)}^\theta \| \nabla v \|_{\dot H^{m-1}(\RR^d)}^{1-\theta}\\
		&\lesssim
		\| v\|_{L^2(\RR^d)}^\theta \| \nabla v \|_{ H^{m-1}(\RR^d)}^{1-\theta}
		\lesssim \| v\|_{L^2(\RR^d)}^\theta \| \nabla v \|_{H^s(\RR^d)}^{1-\theta},
	\end{align*} 
	with $\theta :=\frac{2m-d}{2m}$. Hence the second term on the right-hand side of the above equality can be estimated as follows:\vspace{-0.2cm}
	\begin{equation*}
	\begin{aligned}
		|\langle v\cdot \nabla v,\, v\rangle_{\dot H^s}|
		&=			|\langle v\otimes   v,\, \nabla v\rangle_{\dot H^s}|\\
		&\leq		\|					v	\|_{L^\infty_x	} 
					\|					v	\|_{\dot H^s	}
					\|	\nabla 			v	\|_{\dot H^s	}\\
		&\lesssim	\|					v	\|_{L^2_x		}^{\theta} 
					\| 	\nabla 			v 	\|_{ H^s}^{1-\theta}
					\|					v	\|_{\dot H^s	}
					\|	\nabla 			v	\|_{ H^s	},
	\end{aligned}
	\end{equation*}
	which yields
	\begin{equation}\label{apriori_est_ucdotnablauU}
	\begin{aligned}		
		|\langle v\cdot \nabla v,\, v\rangle_{\dot H^s}|
		&\lesssim	\|					v	\|_{L^2_x		}^{\theta} 
					\|					v	\|_{\dot H^s	}
					\|	\nabla 			v	\|_{ H^s	}^{2-\theta	}\\	
		&\lesssim	\|					v	\|_{L^2_x		}^{\theta	}
					\|					v	\|_{\dot H^s	}^{1-\theta	}	 
					\|					v	\|_{\dot H^s	}^{\theta	}
					\|	\nabla 			v	\|_{	 H^s	}^{2-\theta	}\\	
		&\lesssim	\|					v	\|_{	 H^s	}	 
					\|					v	\|_{\dot H^s	}^{\theta	}
					\|	\nabla 			v	\|_{	 H^s	}^{2-\theta	}.					
	\end{aligned}
	\end{equation}
	Since $s>d/2\geq 1$ then 
	$\|	v	\|_{\dot H^s	}^{\theta	}=\| \nabla 	v	\|_{\dot H^{s-1}}^{\theta	}\leq 	\|	\nabla	v	\|_{ H^s	}^{\theta	}$, thus
	\begin{equation}\label{apriori_est_ucdotnablauU2}	
		|\langle v\cdot \nabla v,\, v\rangle_{\dot H^s}|\leq 
		\|					v	\|_{	 H^s	}	 
		\|	\nabla 			v	\|_{	 H^s	}^2\lesssim \|\nabla v\|_{H^s}^2\|v\|_{H^s}^2 + c_{\beta_4}\|	\nabla v\|_{H^s}^2.
	\end{equation}
	Now, the second term on the right-hand side of \eqref{prop:est1} is 
	\begin{equation*}
		\langle \nabla Q\odot \nabla Q ,\,\nabla v\rangle_{H^s}=
		\langle \nabla Q\odot \nabla Q ,\,\nabla v\rangle_{L^2_x}+
		\langle \nabla Q\odot \nabla Q ,\,\nabla v\rangle_{\dot H^s}.
	\end{equation*}
	We will see that $\langle \nabla Q\odot \nabla Q ,\,\nabla v\rangle_{L^2_x}$ is going to be simplified, while

%-----------------------------------------------------------------------------------------------------------
\begin{equation*}
	\langle		\nabla Q\odot \nabla Q,\,\nabla v\rangle_{\dot H^s}
		\lesssim
	\|	\nabla Q	\|_{L^\infty_x	}
	\|	\nabla Q	\|_{\dot H^s	}
	\|	\nabla v	\|_{\dot H^s	}
	\lesssim
	\|	\nabla Q	\|_{H^s			}^2
	\|	\nabla Q	\|_{H^s			}^2
	+c_{\beta_4}
	\|	\nabla v	\|_{H^s			}^2,
\end{equation*}
Finally, the remaining terms on the right-hand side of \eqref{prop:est1} are controlled as follows:
\begin{equation*}
	\beta_1 	\langle 	\tr\{AQ\}	Q		,\,		\nabla v	\rangle_{H^s}
		\lesssim		\|			A	\|_{H^s}\|	Q	\|_{H^s}^2	\|	\nabla v	\|_{H^s}
	\lesssim	\|	\nabla v	\|_{H^s}^2\|	Q	\|_{H^s}^2,
\end{equation*}

\begin{equation*}
	\beta_5				\langle 				AQ		,\,		\nabla v	\rangle_{H^s}+
	\beta_6				\langle					QA		,\,		\nabla v	\rangle_{H^s}
	\lesssim	\|			A	\|_{H^s}\|	Q	\|_{H^s}	\|	\nabla v	\|_{H^s}
	\lesssim	\|	\nabla v	\|_{H^s}^2\|	Q	\|_{H^s}^2 + c_{\beta_4}\|	\nabla v\|_{H^s}^2,
\end{equation*}

\begin{equation*}
	\frac{\mu_2}{2}		\langle		[\Omega,\,Q]		,\,		\nabla v	\rangle_{H^s}
	\lesssim
				\|			 Q	\|_{H^s}
				\|	\nabla 	 v	\|_{H^s}^2
	\lesssim
				\|	\nabla	 v	\|_{H^s}^2
				\|	 		 Q	\|_{H^s}^2 +
	c_{\beta_4}	\|\nabla 	 v	\|_{H^s}^2 ,
\end{equation*}

\begin{equation*}
	\mu_1\langle [Q,\,\dot Q],\,\nabla v\rangle_{H^s}
	\lesssim
	\|			Q	\|_{H^s}
	\|	\dot 	Q	\|_{H^s}
	\|	\nabla  v	\|_{H^s}
	\lesssim
	\|	\nabla  v	\|_{H^s}^2
	\|			Q	\|_{H^s}^2	+ c_{\mu_1}
	\|	\dot 	Q	\|_{H^s}^2,	
\end{equation*}

\begin{equation*}
	\mu_1\langle [Q,\,[\Omega,\, Q]],\,\nabla v\rangle_{H^s}
	\lesssim
	\|			Q	\|_{H^s}^2
	\|	\nabla v	\|_{H^s}^2.
\end{equation*}
Thus, summarizing the previous estimates we get
\begin{equation}\label{Hs_est1}
\begin{aligned}
	\frac{1}{2}\frac{\dd }{\dd t}&\|	v	\|_{H^s}^2 +\beta_4\|	\nabla v	\|_{H^s}^2 + \frac{\mu_2}{2} \langle \dot Q, \nabla v \rangle_{H^s}
	-L\langle \nabla Q\odot \nabla Q,\,\nabla v\rangle_{L^2_x}\lesssim\\	&
	\lesssim 
	\|	\nabla	v	\|_{H^s}^2
	\big(
		\|			v	\|_{H^s}^2 +
		\|	\nabla	Q	\|_{H^s}^2 +
		\|			Q	\|_{H^s}^2
	\big) + c_{\mu_1}
	\|	\dot 	Q	\|_{H^s}^2 + c_{\beta_4}
	\|	\nabla	v	\|_{H^s}^2+\|\nabla Q\|_{H^s}^2\|\nabla Q\|_{H^s}^2.
\end{aligned}
\end{equation}

Now, let us take the $H^s$-inner product between the order equation and $\dot Q$:
\begin{align*}
	J		\langle		\ddot 		Q	,\,	\dot Q			\rangle_{H^s}  +
	\mu_1	\|	\dot 	Q	\|_{H^s}^2 - 
	\mu_1	\langle		[\Omega,\,Q]	,\,	\dot Q			\rangle_{H^s}  +
	\frac L2 \frac{\dd }{\dd t}\|\nabla Q\|_{H^s}^2-
			L\langle		\Delta Q		,\,	v\cdot\nabla Q	\rangle_{H^s}  =\\=-
	a		\langle		Q				\,\,\dot Q			\rangle_{H^s} 								+
	b		\langle		Q^2				,\,					\dot Q	\rangle_{H^s}  	-
	c		\langle		Q|Q|^2\}	,\,					\dot Q	\rangle_{H^s}   +
	\frac{\tilde\mu_2}{2}	\langle 	A				,\,					\dot Q	\rangle_{H^s}.
\end{align*}
We begin observing that 
\begin{equation*}
	J		\langle		\ddot 		Q	,\,	\dot Q	\rangle_{H^s} = \frac J2 \frac{\dd }{\dd t}\|	\dot Q\|_{H^s}^2 +
			J\langle		v\cdot \nabla\dot Q	,\,	\dot Q	\rangle_{H^s}
\end{equation*}
and that
\begin{equation}\label{apriori_est_uDotQDotQ}
\begin{aligned}
	\langle	&	v\cdot \nabla \dot Q,\,	\dot Q	\rangle_{H^s}
	=
	\underbrace{
	\langle		v\cdot \nabla \dot Q,\,	\dot Q	\rangle_{L^2}}_{=0}+
	\langle		v\cdot \nabla \dot Q,\,	\dot Q	\rangle_{\dot H^s}	
	\\&=
	\langle		[(\sqrt{-\Delta})^s,v\cdot \nabla] \dot Q,\,(\sqrt{-\Delta})^s	\dot Q	\rangle_{L^2}+
	\underbrace{
	\langle		v\cdot \nabla (\sqrt{-\Delta})^s \dot Q,\,(\sqrt{-\Delta})^s	\dot Q	\rangle_{L^2}}_{=0}
	\\&
	\lesssim 
	\| \nabla v \|_{ H^s}\|\dot Q\|_{H^s}^2
	\lesssim 
	\|	\nabla 	v	\|_{H^s}^2
	\|	\dot  Q	\|_{H^s}^2+c_{\mu_1}
	\|	\dot  	Q	\|_{H^s}^2.
\end{aligned}
\end{equation} where for the last inequality we used the commutator estimate from \cite{commutator}:

$$\|[(\sqrt{-\Delta})^s,v\cdot \nabla]B\|_{L^2}=\|(\sqrt{-\Delta})^s [(v\cdot\nabla) B]-(v\cdot\nabla)(\sqrt{-\Delta})^s B\|_{L^2_x}\le c\|\nabla v\|_{H^s}\|B\|_{H^s}.$$

\bigskip
Moreover
\begin{equation}\label{apriori_est_[OmQ]}
\begin{aligned}
	\mu_1	\langle		[\Omega,\,Q]	,\,	\dot Q	\rangle_{H^s}
	\lesssim 
	\|	\nabla 	v	\|_{H^s}
	\|			Q	\|_{H^s}
	\|	\dot  	Q	\|_{H^s}
	\lesssim
	\|	\nabla 	v	\|_{H^s}^2
	\|			Q	\|_{H^s}^2+c_{\mu_1}
	\|	\dot  	Q	\|_{H^s}^2.
\end{aligned}		
\end{equation}
Now, we have
\begin{align*}
	\langle	&\Delta Q,\,	v\cdot\nabla Q	\rangle_{\dot H^s}
	=
	\langle (\sqrt{-\Delta})^s Q_{\alpha\beta, ii},\,(\sqrt{-\Delta})^s\left( v_j Q_{\alpha\beta, j}\right)\rangle_{L^2_x}=\\&-
	\langle (\sqrt{-\Delta})^s Q_{\alpha\beta, i},\,(\sqrt{-\Delta})^s\left(v_{j,i} Q_{\alpha\beta, j}\right)\rangle_{L^2_x}-
	\langle (\sqrt{-\Delta})^s Q_{\alpha\beta, i},\,(\sqrt{-\Delta})^s\left(v_{j} Q_{\alpha\beta, ij}\right)\rangle_{L^2_x}
\end{align*}
and
\begin{equation*}
	|\langle (\sqrt{-\Delta})^s Q_{\alpha\beta, i},\,(\sqrt{-\Delta})^s\left(v_{j,i} Q_{\alpha\beta, j}\right)\rangle_{L^2_x}|
	\leq
	\|	\nabla	Q	\|_{ H^s}^2
	\|	\nabla	v	\|_{ H^s}
	\lesssim
	\|	\nabla 	Q	\|_{H^s}^2
	\|	\nabla	Q	\|_{H^s}^2 + c_{\beta_4}
	\|	\nabla	v	\|_{H^s}^2,	
\end{equation*}
\begin{align*}
	-\langle (\sqrt{-\Delta})^s Q_{\alpha\beta, i},\,(\sqrt{-\Delta})^s(v_{j} Q_{\alpha\beta, ij})\rangle_{L^2_x}
	=
	-\langle   (\sqrt{-\Delta})^s Q_{, i},\, [	(\sqrt{-\Delta})^s,\,	v\cdot\nabla ]Q_{,i}\rangle_{L^2_x}-\\
	-\underbrace{
	 \langle   (\sqrt{-\Delta})^s Q_{, i},\, 	v\cdot\nabla (\sqrt{-\Delta})^sQ_{,i}\rangle_{L^2_x}}_{=0}
	\lesssim
	\|		\nabla	Q											\|_{\dot H^s	}
	\|		[	(\sqrt{-\Delta})^s,\,	v\cdot\nabla ]Q_{,i}	\|_{L^2_x	 	}
	\lesssim\\
	\lesssim
	\|		\nabla	Q		\|_{H^s}^2
	\|		\nabla 	v		\|_{ H^s}
	\lesssim
	\|	\nabla 	Q	\|_{H^s}^2
	\|	\nabla	Q	\|_{H^s}^2 + c_{\beta_4}
	\|	\nabla	v	\|_{H^s}^2,
\end{align*}
which yields
\begin{equation}\label{apriori_est_DeltaQuNablaU}
	\langle 	\Delta Q,\,	v\cdot\nabla Q	\rangle_{\dot H^s}
	\lesssim
	\|	\nabla 	Q	\|_{H^s}^2
	\|	\nabla	Q	\|_{H^s}^2 + c_{\beta_4}
	\|	\nabla	v	\|_{H^s}^2.
\end{equation}
Finally
\begin{equation*}
	-a\langle	\dot Q ,\, Q	\rangle_{H^s} =-\frac{a}{2} \frac{\dd }{\dd t}\| Q\|_{H^s}^2 - a\langle	v\cdot \nabla Q ,\, Q	\rangle_{H^s},
\end{equation*}
with
\begin{equation*}
	\left|a\langle	v\cdot \nabla Q ,\, Q	\rangle_{H^s}\right|\lesssim \| v\|_{H^s}\|\nabla Q\|_{H^s}\| Q \|_{H^s}
	\lesssim
	\|	Q	\|_{H^s}^2\|	v	\|_{H^s}^2+c\|\nabla	Q	\|_{H^s}^2
\end{equation*}

\begin{equation*}
	\left|b\langle \dot Q	,\,	Q^2 \rangle_{H^s}\right|
	\lesssim
	\|	\dot Q	\|_{H^s}\|	Q	\|_{H^s}^2
	\lesssim
	\| 		Q	\|_{H^s}^2
	\| 		Q	\|_{H^s}^2 + c_{\mu_1}
	\|	\dot Q	\|_{H^s}^2
\end{equation*}
and
\begin{equation*}
	\left|c\langle	\dot Q	,\,	Q|Q|^2\rangle_{H^s}\right|
	\lesssim
	\|	\dot Q	\|_{H^s}
	\|		 Q	\|_{H^s}^3
	\leq
	\|		Q	\|_{H^s}^2
	\big(
		\|	\dot 	Q	\|_{H^s}^2+
		\|			Q	\|_{H^s}^2
	\big).
\end{equation*}

Thus, summarizing the previous estimates we get
\begin{equation}\label{Hs_est2}
\begin{aligned}
	& \frac{\dd }{\dd t}
	\big[\frac J2 
			\|	\dot Q	\|_{H^s}^2 +\frac L2 \|	\nabla 	Q	\|_{H^s}^2 + \frac{a}{2}\|Q\|_{H^s}^2	
	\big] + \mu_1
	\|		\dot 	Q	\|_{H^s}^2-
	 L\langle	v\cdot\nabla Q,\,	\Delta Q\rangle_{L^2_x} -\frac{\tilde\mu_2}{2} \langle	A,\,\dot Q\rangle_{H^s}
	\lesssim\\&
	\lesssim
	\big(
		\|			Q	\|_{H^s}^2 +
		\|	\nabla	Q	\|_{H^s}^2 +
		\|	\nabla 	v	\|_{H^s}^2+
		\|v\|_{H^s}^2
	\big)
	\big(
		\|	\dot	Q	\|_{H^s}^2 +
		\|			Q	\|_{H^s}^2 +
		\|	\nabla	Q	\|_{H^s}^2 
	\big)+\\&\quad + c_{\beta_4}
	\|	\nabla 	v	\|_{H^s}^2 + c_{\mu_1}
	\|	\dot	Q	\|_{H^s}^2+c\|\nabla Q\|_{H^s}^2.	
\end{aligned}
\end{equation}

Now, let us consider the $H^s$-inner product between the order tensor equation and $Q/2$, namely
\begin{align*}
	\frac J2 			\langle	\ddot 	Q		,\,		Q	\rangle_{H^s} + 
	\frac{\mu_1}{2}		\langle	\dot 	Q		,\,		Q	\rangle_{H^s} -
	\frac{\mu_1}{2}		\langle	[\Omega,\,Q]	,\,		Q	\rangle_{H^s} +
	\frac{L}{2}\|	\nabla 	Q	\|_{H^s}^2 + \frac{a}{2}\|	Q	\|_{H^s}^2 = \\=
	\frac{b}{2}					\langle		Q^2			,\, 	Q	\rangle_{H^s} -
	\frac{c}{2}					\langle		Q|Q|^2	,\, 	Q	\rangle_{H^s} +
	\frac{\tilde\mu_2}{4}				\langle		A			,\, 	Q	\rangle_{H^s}.
\end{align*}
First, let us observe that
\begin{equation*}
	\langle	\ddot 	Q				,\,		Q	\rangle_{H^s		} =
	\langle	\ddot 	Q				,\,		Q	\rangle_{L^2_x		} +
	\langle	\ddot 	Q				,\,		Q	\rangle_{\dot H^s	}.
\end{equation*}	
We have
\begin{align*}
	\langle	\ddot 	Q									,\,		Q	\rangle_{L^2_x		}& =
	\langle \partial_t \dot 	Q	,\,		Q	\rangle_{L^2_x		} +
	\langle  v\cdot\nabla \dot Q	,\,		Q	\rangle_{L^2_x		} \\&=
	\frac {d}{dt}
	\langle  \dot 	Q	,\,		Q			\rangle_{L^2_x		}-
	\langle  \dot 	Q	,\,	\partial_t	Q	\rangle_{L^2_x		}-
	\langle  \dot Q		,\,	v\cdot\nabla Q	\rangle_{L^2_x		}=
	\frac{d}{dt}
	\langle  \dot 	Q	,\,		Q	\rangle_{L^2_x		} -
	\|	\dot Q	\|_{L^2_x}^2.
\end{align*}
Moreover
\begin{equation*}
						\langle \partial_t	\dot Q,\, 				Q	\rangle_{\dot H^s}=
	\frac{\dd}{\dd t}	\langle				\dot Q,\, 				Q	\rangle_{\dot H^s}-
						\langle				\dot Q,\,	\partial_t 	Q	\rangle_{\dot H^s}
\end{equation*}	
and
\begin{equation}\label{apriori_est1}
\begin{aligned}
	\langle			v\cdot\nabla	\dot 	Q	,\,				Q	\rangle_{\dot H^s}&=
	\langle	(\sqrt{-\Delta})^s		\left(v\cdot\nabla	\dot 	Q\right)	,\,	 (\sqrt{-\Delta})^s		Q	\rangle_{L^2_x}\\&=
	\langle	[(\sqrt{-\Delta})^s,\,		v\cdot\nabla ]	\dot 	Q	,\,	 (\sqrt{-\Delta})^s		Q	\rangle_{L^2_x} -
	\langle	(\sqrt{-\Delta})^s			\dot 	Q	,\,	v\cdot\nabla (\sqrt{-\Delta})^s		Q	\rangle_{L^2_x} \\&=
	\langle	[(\sqrt{-\Delta})^s,\,		v\cdot\nabla ]	\dot 	Q	,\,	 (\sqrt{-\Delta})^s		Q	\rangle_{L^2_x} +\\
	&\quad\quad\quad\quad\quad+
	\langle	(\sqrt{-\Delta})^s			\dot 	Q	,\,	[(\sqrt{-\Delta})^s,\,		v\cdot\nabla ]	Q	\rangle_{L^2_x} -
	\langle				\dot 	Q	,\,	  v\cdot\nabla Q \rangle_{\dot H^s},
\end{aligned}
\end{equation}	
Thus, summarizing, we get
\begin{equation}\label{apriori_est2}
\begin{aligned}
	\frac J2 
	\langle				\ddot 	Q			,\,		&Q	\rangle_{H^s_x		} =
	\frac J2 \frac{\dd}{\dd t}	
	\langle				\dot 	Q			,\, 	Q	\rangle_{H^s} - 
	\frac J2 \|	\dot Q	\|_{H^s}^2 +\\&+
	\frac J2 \langle	[(\sqrt{-\Delta})^s,\,		v\cdot\nabla ]	\dot 	Q	,\,	 (\sqrt{-\Delta})^s		Q	\rangle_{L^2_x} +
	\frac J2 \langle	(\sqrt{-\Delta})^s			\dot 	Q	,\,	[(\sqrt{-\Delta})^s,\,		v\cdot\nabla ]	Q	\rangle_{L^2_x}
\end{aligned}
\end{equation} 
with the estimate
\begin{equation}\label{apriori_est3}
\begin{aligned}
	\frac J2 \langle	[(\sqrt{-\Delta})^s,\,		v\cdot\nabla ]	\dot 	Q	,\,	 (\sqrt{-\Delta})^s		Q	\rangle_{L^2_x}+
	\frac J2 \langle	(\sqrt{-\Delta})^s			\dot 	Q	,\,	[(\sqrt{-\Delta})^s,\,		v\cdot\nabla ]	Q	\rangle_{L^2_x}\\
	\lesssim
	\|	\nabla	v	\|_{	H^s}
	\|	\dot 	Q	\|_{	H^s}
	\|			Q	\|_{ 	H^s}
	\lesssim
	\|	\nabla	v	\|_{ 	H^s}^2
	\|			Q	\|_{ 	H^s}^2+c_{\mu_1}
	\|	\dot 	Q	\|_{ 	H^s}^2.
\end{aligned}
\end{equation}
Furthermore
\begin{equation}\label{apriori_est4}
	\frac J2 \frac{\dd}{\dd t}	\langle				\dot Q,\, 				Q	\rangle_{\dot H^s} =
	\frac J4 \frac{\dd}{\dd t}
	\big[
		\|	\dot Q	+	Q	\|_{\dot H^s}^2 - 
		\|	\dot Q			\|_{\dot H^s}^2 -
		\|				Q	\|_{\dot H^s}^2
	\big].
\end{equation}

\par On the other hand
\begin{equation*}
	\frac{\mu_1}{2}	\langle 		\dot Q		,\,		Q	\rangle_{H^s} =
	\frac{\mu_1}{2}	\langle 	\partial_t Q	,\,		Q	\rangle_{H^s} +
	\frac{\mu_1}{2}	\langle 	v\cdot\nabla Q	,\,		Q	\rangle_{H^s} =
	\frac{\mu_1}{4}
	\frac{\dd}{\dd t}
	\|	Q	\|_{H^s}^2+
	\frac{\mu_1}{2}	\langle 	v\cdot\nabla Q	,\,		Q	\rangle_{H^s},
\end{equation*}
with
\begin{equation*}
	\frac{\mu_1}{2}	\langle 	v\cdot\nabla Q	,\,		Q	\rangle_{H^s}
	\lesssim
	\|	v	\|_{H^s}\|	\nabla	Q	\|_{H^s}\|	Q	\|_{H^s}
	\lesssim
	\|	Q			\|_{H^s}^2
	\|	v			\|_{H^s}^2 +c
	\|	\nabla	Q	\|_{H^s}^2.
\end{equation*}
Then
\begin{equation*}
	\frac{\mu_1}{2}		\langle	[\Omega,\,Q]	,\,		Q	\rangle_{H^s}
	\lesssim
	\|	\nabla	v	\|_{H^s}
	\|			Q	\|_{H^s}^2
	\lesssim
	\|			Q	\|_{H^s}^2
	\|			Q	\|_{H^s}^2+c_{\beta_4}
	\|	\nabla	v	\|_{H^s}^2,
\end{equation*}

\begin{equation*}
	\left|\frac b2					\langle		Q^2			,\, 	Q	\rangle_{H^s}\right|
	\lesssim
	\|			Q	\|_{H^s}^3
	\lesssim
	\|			Q	\|_{H^s}^2
	\|			Q	\|_{H^s}^2+c_a
	\|			Q	\|_{H^s}^2,	
\end{equation*}

\begin{equation*}
	\left|\frac c2				\langle		Q|Q|^2	,\, 	Q	\rangle_{H^s}\right|
	\lesssim
	\|			Q	\|_{H^s}^2
	\|			Q	\|_{H^s}^2
\end{equation*}
and finally
\begin{equation*}
	\left|\frac{ \tilde\mu_2}{ 4}	\langle		A			,\, 	Q	\rangle_{H^s}\right|
	\leq
	\left|\frac{\tilde \mu_2}{ 4}\right|\|\nabla v\|_{H^s}\| Q\|_{H^s}
	\leq
	\frac{|\tilde\mu_2|^2}{32a(1-\varepsilon)}
	\|\nabla 	v	\|_{H^s}^2 +
	\frac{a}{2}(1-\varepsilon)
	\| 			Q	\|_{H^s}^2
\end{equation*}
Then, summarizing the previous estimates, we get
\begin{equation}\label{Hs_est3}
\begin{aligned}
	\frac{\dd}{\dd t}
	\Big[
		\frac J4			\|	\dot Q	+	Q	\|_{H^s}^2 - 
		\frac J4			\|	\dot Q			\|_{H^s}^2 +
		\frac{\mu_1-J}4		\|				Q	\|_{H^s}^2
	\Big]-
	\frac J2 							\|	\dot 	Q	\|_{H^s}^2 + 
	\frac{a}{2}\varepsilon						\| 			Q	\|_{H^s}^2 -\\-
	\frac{|\tilde\mu_2|^2}{32a(1-\varepsilon)}	\|	\nabla 	v	\|_{H^s}^2 +
	\frac{L}{2}\|	\nabla Q\|_{H^s}^2
	\lesssim 
	\Big(
		\|	\nabla 	v	\|_{H^s}^2 +
		\|			Q	\|_{H^s}^2
	\Big)
	\Big(
		\|			Q	\|_{H^s}^2 +
		\|			v	\|_{H^s}^2
	\Big)+\\+c_{\beta_4}
	\|	\nabla 	v	\|_{H^s}^2 +c_{\mu_1}
	\|	\dot 	Q	\|_{H^s}^2 +c_a
	\|			Q	\|_{H^s}^2 + c
	\|	\nabla	Q	\|_{H^s}^2.
\end{aligned}
\end{equation}

Finally, taking the sum between \eqref{Hs_est1}, \eqref{Hs_est2} and \eqref{Hs_est3} and assuming $c_{\beta_4}$, $c_{\mu_1}$, $c$ and $c_a$ small enough, we get for a suitable $\delta>0$ small enough:
\begin{equation*}
\begin{aligned}
	\frac{\dd }{\dd t}
	\Big[
		\frac 12				\|				v	\|_{H^s}^2 + 
		\frac J4				\|	Q	+ \dot	Q	\|_{H^s}^2 +	
		\frac J4				\|		  \dot	Q	\|_{H^s}^2 +
		\big(\frac{a}{2}+\mu_1-\frac J4\big)	\|				Q	\|_{H^s}^2 +
								\frac{L}{2}\|		\nabla	Q	\|_{H^s}^2
	\Big] +\\+
	\big(
		\beta_4 - \frac{|\tilde\mu_2|^2}{32(1-\varepsilon)a}-\delta
	\big)
	\|	\nabla	v	\|_{H^s}^2	+
	\big(
		\mu_1 - \frac{J}{2}-\delta
	\big)
	\|	\dot 	Q	\|_{H^s}^2+
	(\frac{a}{2}\varepsilon-\delta)\|	Q	\|_{H^s}^2 +
	(\frac{L}{2}-\delta) \|	\nabla Q	\|_{H^s}^2\\
%	+\frac{\mu_2}{2}\langle \dot Q,\nabla v\rangle_{H^s}-\frac{\tilde\mu_2}{2}\langle A,\dot Q\rangle_{H^s} \lesssim\\
	\lesssim
	\Big(	
		\|	\nabla	v	\|_{H^s}^2 +
		\|	\dot 	Q	\|_{H^s}^2 +
		\|			Q	\|_{H^s}^2 +
		\|	\nabla	Q	\|_{H^s}^2
	\Big)
	\Big(
		\|				v	\|_{H^s}^2+
		\|				Q	\|_{H^s}^2+
		\|	\dot			Q	\|_{H^s}^2+
		\|	\nabla		Q	\|_{H^s}^2
	\Big),
\end{aligned}	
\end{equation*}
where we have used
\begin{equation*}
	\langle \nabla Q\odot \nabla Q,\,\nabla v\rangle_{L^2_x} + \langle v\cdot \nabla Q,\, \Delta Q\rangle_{L^2_x} =0.
\end{equation*}

\smallskip 

The last estimate allows, under suitable relations on the coefficients, to obtain a differential inequality of the type \eqref{intro:estimate_existence} with $\Phi(t):=\|				v	\|_{H^s}^2+
		\|				Q	\|_{H^s}^2+
		\|	\dot			Q	\|_{H^s}^2+
		\|	\nabla		Q	\|_{H^s}^2$ and $\Psi(t):=\|	\nabla	v	\|_{H^s}^2 +
		\|	\dot 	Q	\|_{H^s}^2 +
		\|			Q	\|_{H^s}^2 +
		\|	\nabla	Q	\|_{H^s}^2$ which allows to control apriori these norms globally in time, for small data (see
also the Lemma~\ref{appx:lemma1} in the Appendix).

\subsection{Strong solutions}

\vspace{1ex}

\smallskip
\par {\bf Proof of Theorem~\ref{intro:thm2}:} We divide the proof into the existence and uniqueness parts. The existence is based on a Friedrichs-type scheme that preserves the structure exhibited in the higher-order energy laws, which allows to construct approximate solutions.  The uniqueness is achieved afterwards through to an $H^s$-type energy estimate.

\smallskip
\textbf{Existence part:}
In order to construct global strong solutions, we use the classical Friedrichs scheme and  obtain estimates similar to the ones in the previous section. We define the mollifying operator
\begin{equation*}
	 J_nf(\xi):=\mathcal{F}^{-1}\left(1_{\{2^{-n}\leq|\xi|\leq 2^n\}}\mathcal{F}f\right),
\end{equation*}

The approximate momentum equation then reads as follows:
\begin{equation}
\begin{aligned}
	&J_n \partial_t v^{(n)} + \LP J_n\left(J_n v^{(n)}\cdot\nabla J_n v^{(n)}\right)+\frac{\beta_4}{2}\Delta J_n \vn =  
	- L\nabla	\cdot	\LP\Big\{ 				J_n \big( 	\nabla J_n \Qn	\otimes	\nabla J_n \Qn \big)	\Big\}	 							+\\& 
	+ \nabla 	\cdot 	\LP\Big\{	\beta_1 	J_n\big( J_n\Qn\trc\{(J_n\Qn J_n\An)\big) + 	\beta_5J_n\big( J_n\An J_n\Qn \big) 		+ 
																								\beta_6J_n\big( J_n\Qn J_n\An \big)\Big\} 	\\\ &
	+ \nabla	\cdot	\LP\Big\{\;\, \frac{\mu_2}{2}\big(J_n\dot Q^{(n)}	- 	J_n[J_n\Omn,\,J_n \Qn]\big)
								+ \mu_1 J_n\big[ J_n\Qn,\,(J_n\dot Q^{(n)} - [J_n\Omn,\,J_n\Qn]\big)] \;\,\,\Big\},
\end{aligned}
\end{equation}
where $\LP$ denotes the Leray projector onto divergence-free vector fields, and we denote $$\Omega^{(n)}_{ij}:=\frac{v^{(n)}_{i,j}-v^{(n)}_{j,i}}{2},A^{(n)}_{ij}:=\frac{v^{(n)}_{i,j}+v^{(n)}_{j,i}}{2},\,i,j=1,\dots,d$$

Similarly, the approximate order tensor equation is
\begin{equation*}
\begin{aligned}
	J J_n  \ddot Q^{(n)} &+ \mu_1J_n  \dot Q^{(n)}  = L\Delta J_n \Qn 
												-aJ_n\Qn 
												+ b  J_n(J_n\Qn J_n\Qn) \\&
												-b \trc\big\{ J_n(J_n\Qn J_n\Qn)\big\}\frac{\Id}{d} 
												+c J_n\Big( J_n\Qn \trc\{ (J_n\Qn J_n\Qn) \}\Big)\\
												&+\frac{\tilde\mu_2}{2} J_n A^{(n)}+\mu_1 J_n[J_n\Omega^{(n)}, J_n Q^{(n)}]
\end{aligned}
\end{equation*}   where we have used the abuse of notation
\begin{equation*}
	\dot f^{(n)} := \partial_t f^{(n)} + J_n (J_n\vn \cdot \nabla J_n f^{(n)} ).
\end{equation*}
The system above can be regarded as an ordinary differential equation in $L^2$ verifying the conditions of the Cauchy-Lipschitz theorem. Thus it
admits a unique maximal solution $(\vn,\,\Qn)$ in $C^1([0,T^n), L^2)$. As we have $(\LP J_n )^2 = \LP J_n $ and $J_n^2 = J_n$, the pair 
$(J_n \vn,\,J_n\Qn)$ is also a solution of the previous system. Hence, by uniqueness we get that $(J_n \vn,\,J_n\Qn)=(\vn,\,\Qn)$, and moreover the pair
$(\vn,\,\Qn)$ belongs to $C^1 ([0, T^n ), H^{\infty})$ and  solves the system
\begin{equation}\label{appx_system}
\left\{
\begin{aligned}
%-------------------------------------------------------------------------------------------------------------------------------------------------
%-------------------------------------------------------------------------------------------------------------------------------------------------
%----------------------------			Appx Mom eq							------------------------------------------------------------
%-------------------------------------------------------------------------------------------------------------------------------------------------
%-------------------------------------------------------------------------------------------------------------------------------------------------
	& \partial_t v^{(n)}+\LP J_n(v^{(n)}\cdot\nabla v^{(n)}) + \frac{\beta_4}{2}\Delta  \vn =  
	-L \nabla	\cdot	\LP\Big\{ 				J_n \big( 	\nabla  \Qn	\otimes	\nabla  \Qn \big)	\Big\}	 							+\\& 
	+ \nabla 	\cdot 	\LP\Big\{	\beta_1 	J_n\big( \Qn\trc\{\Qn \An\}\big) + 	\beta_5J_n\big( \An \Qn \big) 		+ 
																								\beta_6J_n\big( \Qn \An \big)\Big\} 	+\\\quad &
	+ \nabla	\cdot	\LP\Big\{\;\, \frac{\mu_2}{2}\big(\dot Q^{(n)}	- 	J_n[\Omn,\,\Qn]\big)
								+ \mu_1 J_n\big[ \Qn,\,(\dot Q^{(n)} - [\Omn,\,\Qn]\big)\big] \;\,\,\Big\},\\
%-------------------------------------------------------------------------------------------------------------------------------------------------
%-------------------------------------------------------------------------------------------------------------------------------------------------
%----------------------------			Appx Order tensor eq							------------------------------------------------------------
%-------------------------------------------------------------------------------------------------------------------------------------------------
%-------------------------------------------------------------------------------------------------------------------------------------------------
	&J   \ddot Q^{(n)} + \mu_1 \dot Q^{(n)}  = L\Delta  \Qn 
												-a\Qn 
												+ b  J_n(\Qn \Qn) \\&\quad\quad\quad\quad\quad
												-b \trc\big\{ J_n(\Qn \Qn)\big\}\frac{\Id}{d} 
												+c J_n\Big( \Qn \trc\{(\Qn \Qn) \}\Big)\\
												&\quad\quad\quad\quad\quad+\frac{\tilde\mu_2}{2} A^{(n)}+\mu_1 J_n[\Omega^{(n)},  Q^{(n)}]
\end{aligned}
\right.
\end{equation}
%-------------------------------------------------------------------------------------------------------------------------------------------------
%-------------------------------------------------------------------------------------------------------------------------------------------------
%----------------------------			Appx momentum equation estimates				------------------------------------------------------------
%-------------------------------------------------------------------------------------------------------------------------------------------------
%-------------------------------------------------------------------------------------------------------------------------------------------------
Arguing similarly as in the proof of the apriori estimates and taking advantage of the fact that $J_n$ is a self-adjoint operator in $L^2$, in order to obtain similar cancellations we get (for a $\delta>0$ suitably small):
\begin{equation}\label{final_est_Hsn}
\begin{aligned}
	\frac{\dd }{\dd t}
	\Big[
		\frac 12				\|				\vn		\|_{H^s}^2 + 
		\frac J4				\|	\Qn	+ \dot	Q^{(n)}	\|_{H^s}^2 +	
		\frac J4				\|		  \dot	Q^{(n)}	\|_{H^s}^2 +
		\big(\frac{a}{2}+\mu_1-\frac J4\big)	\|				\Qn		\|_{H^s}^2 +
								\frac L2\|		\nabla	\Qn		\|_{H^s}^2
	\Big]\\+
	\big(
		\beta_4 - \frac{|\tilde\mu_2|^2}{32(1-\varepsilon)a}-\delta
	\big)
	\|	\nabla	\vn	\|_{H^s}^2	+
	\big(
		\mu_1 - \frac{J}{2}-\delta
	\big)
	\|	\dot 	Q^{(n)}	\|_{H^s}^2+
	(\frac{a}{2}\varepsilon-\delta)\|	\Qn	\|_{H^s}^2 
	+(\frac L2 -\delta)\|	\nabla \Qn	\|_{H^s}^2 \\
	%+\frac{\mu_2}{2} \langle \dot\Qn,\nabla\vn\rangle_{H^s}-\frac{\tilde\mu_2}{4}\langle A^{(n)},\dot Q^{(n)}\rangle_{H^s}
	\lesssim
	\Big(	
		\|	\nabla	\vn	\|_{H^s}^2 +
		\|	\dot 	Q^{(n)}	\|_{H^s}^2 +
		\|			\Qn	\|_{H^s}^2 +
		\|	\nabla	\Qn	\|_{H^s}^2
	\Big){\scriptstyle \times}\\{\scriptstyle \times}
	\Big(
		\|				\vn	\|_{H^s}^2+
		\|				\Qn	\|_{H^s}^2+
		\|	\dot Q^{(n)}	\|_{H^s}^2+
		\|	\nabla		\Qn	\|_{H^s}^2
	\Big),
\end{aligned}	
\end{equation}

%-------------------------------------------------------------------------------------------------------------------------------------------------
%-------------------------------------------------------------------------------------------------------------------------------------------------
%----------------------------			Conclusion										------------------------------------------------------------
%-------------------------------------------------------------------------------------------------------------------------------------------------
%-------------------------------------------------------------------------------------------------------------------------------------------------
Defining the functions $x(t)$ and $y(t)$ by
\begin{align*}
	x(t)&:=	
		\|	\nabla	\vn	\|_{H^s}^2 +
		\|	\dot 	Q^{(n)}	\|_{H^s}^2 +
		\|			\Qn	\|_{H^s}^2 +
		\|	\nabla	\Qn	\|_{H^s}^2,\\
	y(t)&:=\|				\vn	\|_{H^s}^2+
		\|				\Qn	\|_{H^s}^2+
		\|	\dot	Q^{(n)}	\|_{H^s}^2+
		\|	\nabla		\Qn	\|_{H^s}^2,
\end{align*}
respectively, and thanks to Lemma \ref{appx_lemma_ineq} and inequality \eqref{final_est_Hsn}, we get the following bound:
\begin{align*}
	\sup_{t\in\RR_+}&
	\Big\{	
		\|				\vn(t)	\|_{H^s}^2+
		\|				\Qn(t)	\|_{H^s}^2+
		\|	\dot		Q^{(n)} (t)	\|_{H^s}^2+
		\|	\nabla		\Qn(t)	\|_{H^s}^2
	\Big\} + \\&
	+
	\int_{\RR_+}
	\Big\{
		\|	\nabla	\vn(t)		\|_{H^s}^2 +
		\|	\dot 	Q^{(n)}(t)	\|_{H^s}^2 +
		\|			\Qn(t)		\|_{H^s}^2 +
		\|	\nabla	\Qn(t)		\|_{H^s}^2
	\Big\}\dd t\\
	&\quad\quad\quad\quad\quad\quad\quad\quad\quad\quad\quad
	\lesssim
		\|				v_0	\|_{H^s}^2+
		\|				Q_0	\|_{H^s}^2+
		\|	\dot		Q_0	\|_{H^s}^2+
		\|	\nabla		Q_0	\|_{H^s}^2.
\end{align*}

We claim that these uniform estimates allow us to pass to the limit as, $n$ goes to $\infty$. We  first observe that we can obtain a uniform bound also  for $\partial_t Q^n$ in $L^\infty_{t}H^s$. Indeed
\begin{equation*}
\begin{aligned}
	\sup_{t\in\RR_+}\| \partial_t Q^n \|_{H^s}
	&=
	\sup_{t\in\RR_+}\| \dot Q^n - \vn\cdot\nabla\Qn\|_{H^s}
	\leq
	\| \dot Q^n \|_{L^\infty_t H^s} + \| v^n \|_{L^\infty_t H^s}\| \nabla Q^n \|_{L^\infty_t H^s} \\
	&\lesssim 
	\|				v_0	\|_{H^s}^2+
		\|				Q_0	\|_{H^s}^2+
		\|	\dot		Q_0	\|_{H^s}^2+
		\|	\nabla		Q_0	\|_{H^s}^2.
\end{aligned}
\end{equation*}

Thus, by classical compactness, weak convergence arguments and thanks to the Aubin-Lions lemma, there exists
\begin{equation*}
	Q\in L^\infty_t H^{s+1}\cap L^2_t H^{s+1},\quad v\in L^\infty_t H^s\cap L^2_t H^{s+1},\quad
	\text{and}\quad \omega  \in L^\infty_tH^s\cap L^2_tH^s,
\end{equation*}
such that, up to a subsequence, we have the following convergences
\begin{equation*}
\begin{alignedat}{30}
			&				\Qn 		&&&&\rightarrow	 			&&&&&&	Q 		&&&&&&&&&&&&	\quad\text{strong in}\quad		L^\infty_{t,loc}H^{s+1-\mu}_{loc}		\\
			&		\dot 	Q^{(n)}		&&&&\rightarrow				&&&&&&	\omega	&&&&&&&&&&&&	\quad\text{strong in}\quad		L^\infty_{t,loc}H^{s-\mu}_{loc}		\\
		 	&				\vn			&&&&\rightarrow				&&&&&&	v		&&&&&&&&&&&&	\quad\text{strong in}\quad		L^\infty_{t,loc}H^{s-\mu}_{loc}		\\
	\nabla	&				\vn			&&&&\rightharpoonup	\nabla 	&&&&&&	v		&&&&&&&&&&&&	\quad\text{weak \;  in}\quad		L^2_{t}H^{s}							\\
\end{alignedat}
\end{equation*}
for any suitably small positive constant $\mu$. 

Assuming $s-\mu>d/2$,  we have that   $J_n(\vn\cdot \nabla \Qn)$ 
strongly converges to $v\cdot \nabla Q$ in $L^\infty_{t,loc}H^{s-\mu}_{loc}$, as $n\to\infty$, with  $v\cdot \nabla Q\in L^\infty_t H^s$.
Furthermore
\begin{equation*}
	\partial_t Q = \lim_{n\rightarrow\infty} \partial_t \Qn =\lim_{n\rightarrow\infty} \Big( \dot Q^{(n)} - \vn\cdot\nabla\Qn\Big)
	= \omega -v\cdot \nabla Q \in L^\infty_t H^s,
\end{equation*}
where the limits are considered in the distributional sense. Then, we deduce $\partial_t Q \in L^\infty_t H^s$ and $\omega = \dot Q\in L^\infty_t H^s$.
Finally, the order-tensor equation yields
\begin{align*}
	J\partial_t \dot Q^{(n)} =
	-JJ_n(\vn\cdot \nabla\dot Q^{(n)}) 
	- \mu_1 \dot Q^{(n)} 
	+ \mu_1	J_n [\Omn,\,\Qn]
	+L \Delta \Qn + \frac{\tilde\mu_2}{2} \An -\\
	-a\Qn + b\Big(J_n(\Qn\Qn) - \trc\{(\Qn\Qn)\}\frac{\Id}{d}\Big) - cJ_n(\Qn\trc\{(\Qn\Qn)\},
\end{align*}
hence, observing that 

\begin{align*}
	\| J_n(\vn \cdot \nabla \dot Q^{(n)}) \|_{H^{s-1}}
	&\lesssim
	\| \vn \cdot \nabla \dot Q^{(n)} 		\|_{H^{s-1}}\\
	&=
	\| \nabla \cdot \{\,\vn \otimes  \dot Q^{(n)} \}	\|_{H^{s-1}}\\
	&\lesssim
	\| \vn \otimes  \dot Q^{(n)}	\|_{H^{s}}
	\lesssim 
	\| \vn 			\|_{H^{s}}
	\| \dot Q^{(n)}	\|_{H^{s}},
\end{align*} 
then $\partial_t \dot Q^{(n)}$ belongs to $L^2_{t,loc} H^{s-1}$, with uniformly in $n$ bounded seminorms. Thus
\begin{equation*}
	\partial_t \dot Q^{(n)} \rightharpoonup \partial_t \dot Q\quad \text{weakly in}\quad	L^2_{t,loc} H^{s-1}, 
\end{equation*}
up to a subsequence. Moreover, since $J_n(\vn \otimes  \dot Q^{(n)})$  converges weakly to $v\otimes \dot Q$ in 
$L^2_{t, loc}H^{s}$, then  $J_n(\vn \cdot \nabla \dot Q^{(n)})$   converges weakly to $v\cdot \nabla \dot Q$ in
$L^2_{t, loc}H^{s}$. Then, summarizing we deduce that $\ddot Q^{(n)}$  converges weakly to $\ddot Q$ in 
$L^2_{t, loc}H^{s}$.

\noindent 
These convergences allow us to pass to the limit in the classical solutions of \eqref{appx_system}, deducing that 
$(u,\,Q)$ is classical solution of system \eqref{eqinv:momentum+} and \eqref{eqinv:tensors+}.
%-------------------------------------------------------------------------------------------------------------------------------------------------
%-------------------------------------------------------------------------------------------------------------------------------------------------
%----------------------------			Uniqueness								------------------------------------------------------------
%-------------------------------------------------------------------------------------------------------------------------------------------------
%-------------------------------------------------------------------------------------------------------------------------------------------------

\smallskip
\textbf{Uniqueness part:} We now prove the uniqueness of the strong  solutions previously obtained. Let us consider $(u_1,\,Q_1)$ and $(u_2,\,Q_2)$ to be strong solutions with same initial data. From here on we will use the following notation:
\begin{equation*}
	\delta Q := Q_1-Q_2,\quad \delta \dot Q:= \dot Q_1- \dot Q_2,\quad \delta v:= v_1-v_2,\quad \delta A:=A_1-A_2,\delta \Omega:=\Omega_1-\Omega_2.
\end{equation*}

\smallskip
\noindent
We begin the proof by considering the difference between the order-parameter equations of the two solutions, namely
\begin{align*}
	J\left[ (\delta \dot Q )_t + v_1 \cdot\nabla \delta \dot Q + \delta v\cdot\nabla \dot Q_2\right] + \mu_1 \delta \dot Q  =
	L\Delta \delta Q - a \delta Q  + 
	b \big[ Q_1 \delta Q + \delta Q Q_2 +\\+ \trc\{Q_1 \delta Q + \delta Q Q_2 \} \frac{\Id}{d} \big] 
	-c \delta Q  \trc\{ Q_1^2 \} - c Q_2 \trc\{ \delta Q Q_1 \} - c Q_2 \trc\{ Q_2 \delta Q \}  +\\
	\frac{\tilde\mu_2}{2} \delta A+ \mu_1 [ \Omega_1,\, \delta Q ] + \mu_1 [ \delta \Omega ,\,Q_2 ].
\end{align*}
We multiply by $\delta \dot Q $, take the trace and integrate over $\RR^d$ to get:
\begin{equation}\label{uniqueness_energy_Q}
\begin{aligned}
	\frac{\dd}{\dd t}
	\Big[
		\frac{J}{2}\| \delta \dot Q \|^2_{L^2_x} +\frac L2 \| \nabla \delta Q \|_{L^2_x}^2 + \frac{a}{2} \| \delta Q \|_{L^2_x}^2
	\Big]	
	 + \mu_1 \| \delta \dot Q \|_{L^2_x}^2 = 
	 		L\langle \Delta \delta Q 					,\,  v_1\cdot \nabla \delta Q	\rangle_{L^2_x}+\\
	 + 		L\langle \Delta \delta Q 					,\,  \delta v\cdot \nabla Q_2	\rangle_{L^2_x}
	 -		J\langle v_1 \cdot\nabla \delta \dot Q		,\,	\delta \dot Q			\rangle_{L^2_x}
	 -		J\langle  \delta v \cdot\nabla \dot Q_2		,\,	\delta \dot Q			\rangle_{L^2_x}- \\
	 -	a	\langle  \delta	 	Q						,\,	
	 						 v_1\cdot \nabla \delta Q+	\delta v \cdot \nabla Q_2	\rangle_{L^2_x}
	 +	b	\langle  Q_1 \delta Q + \delta Q Q_2		,\,	\delta \dot Q				\rangle_{L^2_x}-\\
	 -  c	\langle  	\delta Q  \trc\{ Q_1^2 \} 
	 					+Q_2 \trc\{ \delta Q Q_2 \} 
	 					+ Q_2 \trc\{ Q_2 \delta Q \}	,\,	\delta \dot Q				\rangle_{L^2_x}+\\
	 +		\frac{\tilde\mu_2}{2}\langle \delta A							,\,	\delta \dot Q				\rangle_{L^2_x}
	 +\mu_1	\langle	[ \Omega_1,\, \delta Q ] + 
	 				[ \delta \Omega ,\,Q_2 ]			,\,	\delta \dot Q				\rangle_{L^2_x}.
\end{aligned}
\end{equation} 
We now estimate  each term on the right-hand side. First we remark that
\begin{align*}
	 \langle \Delta \delta Q 							,\,  v_1\cdot \nabla \delta Q			\rangle_{L^2_x} &= 
	 \langle \delta Q_{\alpha\beta,\, jj}				,\, (v_{1})_{i}\delta Q_{\alpha\beta,i}	\rangle_{L^2_x} \\&=
	-\langle \delta Q_{\alpha\beta,\, j}				,\,  (v_{1})_{i,j}\delta Q_{\alpha\beta,i}	\rangle_{L^2_x} 
	\underbrace{
	-\langle \delta Q_{\alpha\beta,\, j}				,\,  (v_{1})_{i}\delta Q_{\alpha\beta,ij}	\rangle_{L^2_x}}_{=0},
\end{align*}
where for the second equality we have integrated by parts. Then we obtain
\begin{align*}
	 \langle \Delta \delta Q 							,\,  v_1\cdot \nabla \delta Q			\rangle_{L^2_x}
	\lesssim 
	\| 	\nabla 	\delta 	Q 	\|_{L^2_x}
	\| 	\nabla	\delta 	Q 	\|_{L^2_x} 
	\|	\nabla 			v_1 \|_{L^\infty_x}
	\lesssim
	\|	\nabla 			v_1 \|_{H^s		}
		\| 	\nabla 	\delta 	Q 	\|_{L^2_x	}^2,
\end{align*}
Similarly, we can proceed integrating by parts also for the second term, namely
\begin{align*}
	\langle \Delta 	\delta Q 					,\,  \delta v\cdot \nabla Q_2							\rangle_{L^2_x} 
	&= 
	\langle 		\delta Q_{\alpha\beta,jj}	,\,  \delta v_{i}\cdot  (Q_{2})_{ \alpha\beta , \,i}	\rangle_{L^2_x} \\
	&=
	\underbrace{
	-\langle 		\delta Q_{\alpha\beta,j}	,\,  \delta v_{i, j}\cdot  (Q_{2})_{ \alpha\beta , \,i}	\rangle_{L^2_x} }_{\mathcal{A}}
	\underbrace{
	-\langle 		\delta Q_{\alpha\beta,j}	,\,  \delta v_{i}\cdot  (Q_{2})_{ \alpha\beta , \,ij}	\rangle_{L^2_x}}_{\mathcal{B}}.
\end{align*}
First, we control $\mathcal{A}$ using a standard estimate:
\begin{align*}
	\mathcal{A}\lesssim \| \nabla \delta Q \|_{L^2_x} \| \nabla \delta v \|_{L^2_x} \| \nabla Q_2 \|_{L^\infty_x}
	\lesssim 
	\| \nabla 			Q_2 \|_{H^s		}^2
	\| \nabla \delta	Q	\|_{L^2_x	}^2 +c_{\beta_4} 
	\| \nabla \delta	 v	\|_{L^2_x}^2.
\end{align*}
%------------------------------------------------------------------------------------------------------------------------------------------
%------------------------------------------------------------------------------------------------------------------------------------------
%------------------------------------------------------------------------------------------------------------------------------------------

The term $\mathcal{B}$ requires a more careful analysis. First, we define the parameter $\theta$ in $(0,1/2]$ as the minimum between $1/2$ and 
$s-d/2$. Thus, since $\Delta Q_2$ belongs to $L^2(\RR_+,H^{s-1}(\RR^d))$, then it belongs also to 
$L^2(\RR_+,H^{\theta+d/2-1}(\RR^d))$. We will make use of the following Sobolev embeddings:\vspace{-0.2cm}
\begin{equation}\label{QIn_embeddings}
\begin{alignedat}{4}
			H^{s-1}			(\RR^d)&\hookrightarrow H^{\theta+d/2-1}(\RR^d)		&&\hookrightarrow 	L^\frac{d}{1-\theta}(\RR^d),\\
		 	H^1				(\RR^d)&\hookrightarrow	 L^\frac{2d}{d-2(1-\theta)}(\RR^d) &&
\end{alignedat}\vspace{-0.2cm}
\end{equation}
Then $\mathcal{B}$ is bounded by
\begin{align*}
	\mathcal{B} &\lesssim
	\|	\nabla		\delta  Q 	\|_{L^2_x}
	\|				\delta  v 	\|_{L^\frac{2d}{d-2(1-\theta)}_x}
	\| 	\Delta				Q_2	\|_{L^\frac{d}{1-\theta}_x}
	\lesssim
	\|	\nabla		\delta  Q 	\|_{L^2_x}
	\|				\delta  v 	\|_{H^1}
	\| 	\Delta				Q_2	\|_{H^{\theta+\frac{d}{2}-1}}\\
	&\lesssim
	\|	\nabla		\delta  Q 	\|_{L^2_x}
	\|				\delta  v 	\|_{L^2}
	\| 	\Delta				Q_2	\|_{H^{s-1}} + 
	\|				\delta  Q 	\|_{L^2_x}
	\|	\nabla		\delta  v 	\|_{L^2}
	\| 	\Delta				Q_2	\|_{H^{s-1}}\\
	&\lesssim
	\| \nabla Q_2 \|_{H^s}^2
	\big(
		\|		\delta  Q 	\|_{L^2_x}^2+
		\|				\delta  v 	\|_{L^2_x}^2
	\big)
	+c_{\beta_4}
	\|	\nabla		\delta  v 	\|_{L^2_x}^2
	+c
	\|	\nabla		\delta  Q 	\|_{L^2_x}^2.
\end{align*}
Summarizing, the second term is estimated as follows:
\begin{equation*}
	\langle \Delta 	\delta Q 					,\,  \delta v\cdot \nabla Q_2							\rangle_{L^2_x} 
	\lesssim
	\| \nabla Q_2 \|_{H^s}^2
	\big(
		\|	\nabla		\delta  Q 	\|_{L^2_x}^2+\|\delta Q\|_{L^2_x}^2+
		\|				\delta  v 	\|_{L^2_x}^2
	\big)
	+c_{\beta_4}
	\|	\nabla		\delta  v 	\|_{L^2_x}^2
	+c
	\|	\nabla		\delta  Q 	\|_{L^2_x}^2.
\end{equation*}
Now, let us observe that $\langle	v_1\cdot \nabla \delta \dot Q ,\, \delta \dot Q \rangle_{L^2_x}	 = 0$ because of the free divergence condition of $v_1$. Moreover, still recalling the embeddings \eqref{QIn_embeddings}, we have
\begin{align*}
	\langle \delta v \cdot \nabla \dot Q_2,\,\delta \dot Q \rangle_{L^2_x}
	&\lesssim
	\|	\delta v 		\|_{L^\frac{2d}{d-2(1-\theta)}_x	}
	\|	\nabla \dot Q_2	\|_{L^\frac{d}{1-\theta}_x		}
	\| \delta\dot Q			\|_{L^2_x						}
	\lesssim	
	\|	\delta v 			\|_{H^1							}
	\|	\nabla \dot Q_2		\|_{H^{s-1}						}		
	\| \delta\dot Q				\|_{L^2_x						}\\
	&\lesssim 
	\|	\dot Q_2			\|_{H^s							}
	\|	\delta v			\|_{L^2							}	
	\| \delta\dot Q				\|_{L^2_x						}+
	\|	\dot Q_2			\|_{H^s							}
	\|	\nabla 	\delta v	\|_{L^2							}	
	\| \delta\dot Q				\|_{L^2_x						}\\
	&\lesssim 
	\| \dot Q_2 \|_{H^s}^2
	\big(
		\| \delta v \|_{L^2_x}^2+ 
		\| \delta\dot Q \|_{L^2_x}^2
	\big) +c_{\beta_x}
	\| \nabla  \delta v 	\|_{L^2_x}^2 + c_{\mu_1}
	\| \delta\dot  Q	 	\|_{L^2_x}^2.
\end{align*}

%------------------------------------------------------------------------------------------------------------------------------------------
%------------------------------------------------------------------------------------------------------------------------------------------
%------------------------------------------------------------------------------------------------------------------------------------------
%------------------------------------------------------------------------------------------------------------------------------------------
The remaining terms can easily controlled by the H\"older inequality and the Sobolev embedding $H^s\hookrightarrow L^\infty_x$. First the terms related to the parameter $a$ fulfill
\begin{align*}
	\langle \delta Q, v_1\cdot \nabla \delta Q \rangle_{L^2_x} 
	&\lesssim 
	\| \delta Q				 \|_{L^2_x}
	\|	v_1					\|_{L^\infty}
	\|	\nabla \delta Q 	\|_{L^2_x}
	\lesssim
	\|	v_1					\|_{H^s}
	\big(
		\| \delta Q				 	\|_{L^2_x}^2+
		\|	\nabla \delta Q 		\|_{L^2_x}^2
	\big),\\
	\langle \delta Q, \delta v\cdot\nabla Q_2 \rangle_{L^2_x}
	&\lesssim
	\| \delta Q \|_{L^2_x}
	\| \delta v \|_{L^2_x}
	\| \nabla Q_2 \|_{L^\infty_x}
	\lesssim 
	\| \nabla Q_2 \|_{H^s}
	\big(
		\| \delta Q \|_{L^2_x}^2+
		\| \delta v \|_{L^2_x}^2
	\big),
\end{align*}
The terms related to $b$ can be bounded as follows
\begin{align*}
	\langle  Q_1 \delta Q + \delta Q Q_2		,\,	\delta \dot Q				\rangle_{L^2_x}
	&\lesssim
	\| (Q_1,\,Q_2 )		\|_{L^\infty_x}
	\| \delta	Q		\|_{L^2_x}
	\| \delta \dot	Q	\|_{L^2_x}\\
	&\lesssim
	\| (Q_1,\,Q_2 )		\|_{H^s		}^2
	\| \delta	Q		\|_{L^2_x	}^2+c_{\mu_1}
	\| \delta \dot	Q	\|_{L^2_x	}^2
\end{align*}
and finally the one multiplied by $c$ is estimated by
\begin{align*}
	\langle  	\delta Q  \trc\{ Q_1^2 \} 
	 					+Q_2 \trc\{ \delta Q Q_2 \} 
	 					+ Q_2 \trc\{ Q_2 \delta Q \}	,\,	\delta \dot Q				\rangle_{L^2_x}
	 \lesssim
	 \| (Q_1,\,Q_2 )		\|_{H^s		}^2
	 \big(
	 	\| \delta		Q		\|_{L^2_x	}^2+
	 	\| \delta \dot	Q		\|_{L^2_x	}^2
	 \big).
\end{align*}
It remains to control the terms related to  $\mu_1$ and $\mu_2$ which can be handled through
\begin{align*}
	\langle \delta A							,\,	\delta \dot Q				\rangle_{L^2_x}
	\lesssim
	\| \delta A 		\|_{L^2_x}
	\| \delta \dot Q	\|_{L^2_x}
	\lesssim
	\| \delta \dot Q	\|_{L^2_x}^2 +c_{\beta_4}
	\| \nabla \delta v 	\|_{L^2_x}^2
\end{align*}
and
\begin{align*}
	\langle	[ \Omega_1,\, \delta Q ] + 
	 				[ \delta \Omega ,\,Q_2 ]			,\,	\delta \dot Q				\rangle_{L^2_x}
	\lesssim
	\big(
		\| \nabla 	v_1 	\|_{H^s	}^2+
		\| 			Q_2 	\|_{H^s	}^2
	\big)
	\big(
		\| \delta		 Q		\|_{L^2_x}^2+
		\| \delta \dot 	Q		\|_{L^2_x}^2
	\big) + \\+
	c_{\mu_1}
	\| \delta \dot Q \|_{L^2_x}^2 + 
	c_{\beta_4}
	\| \nabla \delta v \|_{L^2_x}^2.
\end{align*}
Using all the previous estimates in the equality \eqref{uniqueness_energy_Q}, we obtain
\begin{equation}\label{uniq_part1}
\begin{aligned}
	\frac{\dd}{\dd t}
	\Big[
		\frac{J}{2}\| \delta \dot Q \|^2_{L^2_x} +\frac{L}{2} \| \nabla \delta Q \|_{L^2_x}^2 + \frac{a}{2} \| \delta Q \|_{L^2_x}^2
	\Big]	
	 + \mu_1 \| \delta \dot Q \|_{L^2_x}^2 
	 \lesssim 
	 \Big(
	 	1+
	 	\| 	Q_2			\|_{	H^s	}^2+
	 	\| 	\nabla v_1	\|_{	H^s }^2+\|v_1\|_{H^s}^2+\\+
		\|\dot Q_2\|_{H^s}^2+
	 	\| 	\nabla Q_2	\|_{	H^s }^2+
	 	\| 	Q_1			\|_{	H^s	}^2
	 \Big)
	 \Big(
	 	\|	\delta v			\|_{L^2_x}^2+
	 	\|	\delta \dot Q	\|_{L^2_x}^2+
	 	\|	\delta Q			\|_{L^2_x}^2+
	 	\|	\nabla	\delta Q	\|_{L^2_x}^2 
	 \Big)+\\+c_{\beta_4}
	 \| \nabla \delta v \|_{L^2_x}^2 + 
	 c_{\mu_1} \| \delta \dot Q \|_{L^2_x}.
\end{aligned}
\end{equation}
Now let us consider the difference between the momentum equations of the two solutions, namely
\begin{equation}\label{uniqueness_energy_u}
\begin{aligned}
	\partial_t \delta v + v_1 \cdot \nabla \delta v + \delta v \cdot \nabla v_2 - \frac{\beta_4}{2}\Delta \delta v = 
	- L\nabla\cdot \Big\{ \nabla \delta Q \otimes \nabla Q_1 + \nabla Q_2 \otimes \nabla\delta Q \Big\}-\\
	+\beta_1 \nabla\cdot \Big\{ \trc\{ \delta QA_1 \} Q_1 + \trc\{ Q_2 \delta A \} Q_1+\trc\{ Q_2A_2 \}\delta  Q \Big\} 
	+\beta_5 \nabla\cdot \big\{	A_1 \delta Q + \delta A Q_2 		\big\}	+ \\
	+\beta_6 \nabla\cdot \big\{	\delta Q  A_1+ Q_2	\delta A 	\big\}	\Big\}
	+ \frac{\mu_2}{2} \nabla \cdot \Big\{ \delta \dot Q - [ \delta \Omega,\, Q_1] -  [  \Omega_2,\, \delta Q] \Big\}+\\
	+\mu_1	\nabla \cdot \big\{ 	[\delta Q	,\, (\dot Q_1 - [ \Omega_1,\, Q_1])] +
		    						[ 		Q_2	,\,	(\delta \dot Q - [ \delta \Omega,\, Q_1] -[ \Omega_2,\, \delta Q])]\Big\}.
\end{aligned}
\end{equation}
We proceed similarly as before, multiplying scalarly by $\delta v$ and integrating everything over $\RR^d$, and by parts, to obtain
\begin{equation}
\begin{aligned}
	\frac{1}{2}\frac{\dd}{\dd t}\| \delta v \|_{L^2_x}^2 + \frac{\beta_4}{2}\| \nabla \delta v \|_{L^2_x}^2  = 
	L \langle		\nabla \delta Q \otimes \nabla Q_1 + 
				\nabla Q_2 \otimes \nabla \delta Q				,\, \nabla	\delta v \rangle_{L^2_x}+\\
	-\beta_1
	 \langle		\trc\{ \delta QA_1 \} 		Q_1 +
				\trc\{ 		Q_2A_2 \}\delta Q 				,\, \nabla	\delta v \rangle_{L^2_x}
	+\beta_1
	 \langle		\trc\{  Q_2\delta A \} 		Q_1			,\, \nabla	\delta v \rangle_{L^2_x}-\\
	-\beta_5
	\langle			A_1 \delta Q + \delta A Q_2 				,\, \nabla	\delta v \rangle_{L^2_x}
	-\beta_6
	\langle				\delta Q  A_1+ Q_2	\delta A 		,\, \nabla	\delta v \rangle_{L^2_x}
	-\frac{\mu_2}{2}
	\langle 			\delta \dot Q							,\, \nabla	\delta v \rangle_{L^2_x}+\\
	+\frac{\mu_2}{2}
	\langle [ \delta \Omega,\, Q_1] +[  \Omega_2,\, \delta Q],\, \nabla	\delta v \rangle_{L^2_x}
	-\mu_1
	\langle			[\delta Q	,\,\dot Q_1 ]				,\, \nabla	\delta v \rangle_{L^2_x}
	-\mu_1
	\langle			[ Q_2	,\, \delta \dot Q ]				,\, \nabla	\delta v \rangle_{L^2_x}-\\
	+\mu_1	
	\langle		[ Q_2	,\,[ \delta \Omega,\, Q_1]
							+[ \Omega_2,\, \delta Q]]		,\, \nabla	\delta v \rangle_{L^2_x}
	+\mu_1 \langle \left[\delta Q, [\Omega_1,Q_1]],\nabla \delta v\right]\rangle_{L^2_x}\\						
	-\langle				v_1	\cdot	\nabla	\delta	v		,\,			\delta v \rangle_{L^2_x}
	-\langle		\delta 	v	\cdot	\nabla			v_2		,\,			\delta v \rangle_{L^2_x}	,
\end{aligned}
\end{equation}
We proceed by estimating each term on the right-hand side. First we have
\begin{align*}
	 \langle		\nabla \delta Q \otimes \nabla Q_1 + 
				\nabla Q_2 \otimes \nabla \delta Q				,\, \nabla	\delta v \rangle_{L^2_x} 
	&\lesssim
	\Big(	
		\| \nabla Q_1 \|_{L^\infty_x} +
		\| \nabla Q_2 \|_{L^\infty_x} 
	\Big)
	\| \nabla \delta Q \|_{L^2_x}	
	\| \nabla \delta v\|_{L^2_x}\\
	&\lesssim
	\Big(	
		\| \nabla Q_1 \|_{H^s}^2 +
		\| \nabla Q_2 \|_{H^s}^2 
	\Big)
	\| \nabla \delta Q \|_{L^2_x}^2+c_{\beta_4}	
	\| \nabla \delta v\|_{L^2_x}^2,
\end{align*}
while the terms concerning $\beta_1$ are handled by
\begin{align*}
	\langle		\trc\{ \delta QA_1 \} 		Q_1 &+
				\trc\{ 		Q_2A_2 \}\delta Q 				,\, \nabla	\delta v \rangle_{L^2_x}
	\lesssim\\
	&\lesssim
	\Big(
		\| \nabla u_1 \|_{L^\infty_x} \| Q_1 \|_{L^\infty_x} +
		\| \nabla u_2 \|_{L^\infty_x} \| Q_2 \|_{L^\infty_x}
	\Big)
	\|  \delta 			Q \|_{L^2_x}
	\|	\nabla\delta 	v \|_{L^2_x}\\
	&\lesssim
	\Big(
		\| \nabla v_1 \|_{H^s}^2 \| Q_1 \|_{H^s}^2 \;+\;
		\| \nabla v_2 \|_{H^s}^2 \| Q_2 \|_{H^s}^2
	\Big)
	\|  \delta 			Q \|_{L^2_x}^2+c_{\beta_4}
	\|	\nabla\delta 	v \|_{L^2_x}^2,
\end{align*}
and
\begin{align*}
	 \langle		\trc\{  Q_2\delta A \} 		Q_1			,\, \nabla	\delta v \rangle_{L^2_x}	
	 \lesssim
	 \| Q_1 \|_{L^\infty_x}\| Q_2 \|_{L^\infty_x} \| \nabla \delta v \|_{L^2_x}^2 
	 \lesssim
	 \| Q_1 				\|_{H^s}
	 \| Q_2 				\|_{H^s}
	 \| \nabla \delta v \|_{L^2_x}^2.
\end{align*}
Now, we bound the terms related to $\beta_5$ and $\beta_6$ as follows:
\begin{align*}
	\langle			A_1 \delta Q + \delta A Q_2 				,\, \nabla	\delta v \rangle_{L^2_x}
	&\lesssim
	\| \nabla v_1 		\|_{L^\infty	}
	\| \delta Q	  		\|_{L^2_x	}
	\| \nabla \delta v 	\|_{L^2_x	} + 
	\|  Q_2				\|_{L^\infty	}
	\| \nabla \delta v	\|_{L^2_x	}^2\\
	&\lesssim
	\| \nabla v_1 		\|_{H^s		}^2
	\| \delta Q	  		\|_{L^2_x	}^2+
	\big(
		c_{\beta_4}+
		\| Q_2\|_{H^s}
	\big)
	\| \nabla \delta v 	\|_{L^2_x	}^2,
\end{align*}
\begin{align*}
	\langle				\delta Q  A_1+ Q_2	\delta A 		,\, \nabla	\delta v \rangle_{L^2_x}
	&\lesssim
	\|	\delta Q			\|_{L^2_x		}
	\|	\nabla v_1		\|_{L^\infty_x	}
	\|	\nabla \delta v	\|_{L^2_x		} +
	\|		   Q_2		\|_{L^\infty_x	}
	\|	\nabla \delta v	\|_{L^2_x		}^2\\
	&\lesssim
	\| \nabla v_1 		\|_{H^s		}^2
	\| \delta Q	  		\|_{L^2_x	}^2+
	\big(
		c_{\beta_4}+
		\| Q_2\|_{H^s}
	\big)
	\| \nabla \delta v 	\|_{L^2_x	}^2.
\end{align*}
Now, we move on and  bound the terms related to $\mu_2$ by
$$	\langle 			\delta \dot Q							,\, \nabla	\delta v \rangle_{L^2_x}
	\lesssim
	\| \delta \dot Q	\|_{L^2_x}^2+c_{\beta_4}\| \nabla \delta v 	\|_{L^2_x	}^2,
$$
\begin{align*}
	\langle [ \delta \Omega,\, Q_1] +[  \Omega_2,\, \delta Q],\, \nabla	\delta v \rangle_{L^2_x}
	&\lesssim
	\| 					Q_1	\|_{L^\infty_x	}
	\| \nabla \delta 	v	\|_{L^2_x		}^2+
	\| \nabla			v_2 \|_{L^\infty_x	}
	\|		  \delta		Q	\|_{L^2_x		}
	\| \nabla \delta 	v	\|_{L^2_x		}\\
	&\lesssim
	\| \nabla			v_2 \|_{H^s_x	}^2
	\|		  \delta		Q	\|_{L^2_x		}^2+	
	\big(
		c_{\beta_4}+
		\| 					Q_1	\|_{H^s_x	}
	\big)
	\| \nabla \delta v	\|_{L^2_x}^2,
\end{align*}
while the terms related to $\mu_1$ can be handled by
\begin{align*}
	\langle			[\delta Q	,\,\dot Q_1 ]				,\, \nabla	\delta v \rangle_{L^2_x}
	&\lesssim
	\|	\delta Q			\|_{L^2_x		}
	\|	\dot Q_1			\|_{L^\infty_x	}
	\|	\nabla \delta v	\|_{L^2_x		}
	\lesssim
	\|	\dot Q_1			\|_{H^s			}^2
	\|	\delta Q			\|_{L^2_x		}^2+c_{\beta_4}
	\|	\nabla \delta v	\|_{L^2_x		}^2,\\
	\langle			[ Q_2	,\, \delta \dot Q ]				,\, \nabla	\delta v \rangle_{L^2_x}
	&\lesssim
	\| Q_2 				\|_{L^\infty_x}
	\| \delta \dot Q 	\|_{L^2_x}
	\| \nabla \delta v 	\|_{L^2_x}
	\lesssim
	\| Q_2 				\|_{H^s		}^2
	\| \delta \dot Q 	\|_{L^2_x	}^2+c_{\beta_4}
	\| \nabla \delta v 	\|_{L^2_x	}^2
\end{align*}
and also
\begin{align*}
	\langle		[ Q_2	,\,[ \delta \Omega,\, Q_1]
							&+[ \Omega_2,\, \delta Q]]		,\, \nabla	\delta v \rangle_{L^2_x}\\
	&\lesssim
	\|			Q_2			\|_{L^\infty_x	}
	\|			Q_1			\|_{L^\infty_x	}
	\|	\nabla	\delta v		\|_{L^2_x		}^2 +
	\|			Q_2			\|_{L^\infty_x	}
	\|	\nabla	v_2			\|_{L^\infty_x	}
	\|	\delta	Q			\|_{L^2_x		}
	\|	\nabla	\delta v		\|_{L^2_x		}\\
	&\lesssim
	\|			Q_2			\|_{H^s_x	}^2
	\|	\nabla	v_2			\|_{H^s_x	}^2
	\|	\delta	Q			\|_{L^2_x	}^2 +
	\Big(
		c_{\beta_4} +
		\|			Q_2		\|_{H^s_x	}
		\|			Q_1		\|_{H^s_x	}
	\Big)
	\| \nabla \delta v \|_{L^2_x}^2,
\end{align*}

\begin{align*}
\langle [\delta Q,[\Omega_1,Q_1]],\nabla\delta v\rangle_{L^2_x}&\le \|Q_1\|_{L^\infty}\|\nabla v_1\|_{L^\infty}\|\delta Q\|_{L^2_x}\|\nabla \delta v\|_{L^2_x}\non\\
&\le \|Q_1\|_{H^s}^2\|\nabla v_1\|_{H^s}^2\|\delta Q\|_{L^2}^2+c_{\beta_4}\|\nabla\delta v\|_{L^2}^2
\end{align*}
Finally, let us remark that $	\langle		v_1	\cdot	\nabla	\delta	v		,\,			\delta v \rangle_{L^2_x}=0$ and
\begin{align*}
	\langle		\delta 	v	\cdot	\nabla			v_2		,\,			\delta v \rangle_{L^2_x}
	\lesssim
	\| 	\nabla v_2	\|_{L^\infty_x	}
	\|	\delta v 	\|_{L^2_x		}^2
	\lesssim
	\| 	\nabla v_2	\|_{H^s			}
	\|	\delta v 	\|_{L^2_x		}^2.
\end{align*}
Thus, summarizing all the previous estimates and using them in \eqref{uniqueness_energy_u} we get
\begin{equation}\label{uniq_part2}
\begin{aligned}
	\frac{1}{2}\frac{\dd}{\dd t}\|\delta v \|_{L^2_x}^2 + \frac{\beta_4}{2}\| \nabla \delta v \|_{L^2_x}^2  
	\lesssim &
	\Big\{
		1+
		\| \nabla v_2 \|_{H^s}+
		\| \nabla v_1 \|_{H^s}^2+
		\| \nabla v_2 \|_{H^s}^2+
		\| \nabla Q_1 \|_{H^s}^2\\ &+
		\|Q_2\|_{H^s}^2+\| \nabla Q_2\|_{H^s}^2+ 
		\| \nabla v_1 \|_{H^s}^2 \| Q_1 \|_{H^s}^2+
		\| \nabla v_2 \|_{H^s}^2 \| Q_2 \|_{H^s}^2+
		\| \dot	  Q_1 \|_{H^s}^2
	\Big\}\\
	&\times\Big\{
		\| 			\delta v \|_{L^2_x}^2+
		\| \nabla 	\delta Q \|_{L^2_x}^2+
		\| 		 	\delta Q \|_{L^2_x}^2+
		\| \delta\dot Q	 \|_{L^2_x}^2
	\Big\}+\\
	&+\Big\{
		c_{\beta_4} 
		+
		\|			Q_2		\|_{H^s_x	}
		\|			Q_1		\|_{H^s_x	}+
		\|			Q_1		\|_{H^s_x	}+
		\|			Q_2		\|_{H^s_x	}
	\Big\}
	\| \nabla \delta v \|_{L^2_x}^2
\end{aligned}
\end{equation}
Now, defining the functions $\Psi= \Psi(t)$ and $f=f(t)$ by
\begin{align*}
	\Psi	&:= 
		\frac 12\|\delta v \|_{L^2_x}^2+
		\frac J2\| \delta \dot Q \|^2_{L^2_x} +\frac{L}{2} \| \nabla \delta Q \|_{L^2_x}^2 + \frac{a}{2} \| \delta Q \|_{L^2_x}^2\\
	f 	&:=	\Big\{
		1+
		\| Q_1 \|_{H^s}^2
		\| \nabla v_2 \|_{H^s}^2+
		\| \nabla v_1 \|_{H^s}^2+
		\| \nabla v_2 \|_{H^s}^2+
		\| \nabla Q_1 \|_{H^s}^2+
		\|Q_2\|_{H^s}^2+
		\\ &\quad\quad+
		\| \nabla Q_2\|_{H^s}^2+ 
		\| \nabla v_1 \|_{H^s}^2 \| Q_1 \|_{H^s}^2+
		\| \nabla v_2 \|_{H^s}^2 \| Q_2 \|_{H^s}^2+
		\| \dot	  Q_1 \|_{H^s}^2 +\\
		& \quad\quad +\|\dot Q_2\|_{H^s}^2+\|Q_1\|_{H^s}^2+\|v_1\|_{H^s}^2
	\Big\},
\end{align*}
and observing that $f \in L^1_{loc}(\RR_+)$, we finally take the sum between \eqref{uniq_part1} and \eqref{uniq_part2}, obtaining
\begin{align*}
	\frac{\dd}{\dd t}\Psi &+ \mu_1\| \dot \delta Q \|_{L^2_x} + \frac{\beta_4}{2}\| \nabla \delta v  \|_{L^2_x}   
	\lesssim
	f \Psi  + c_{\mu_1}\| \delta  \dot Q \|_{L^2_x} + \\&+
	\Big\{
		c_{\beta_4} +
		\|			Q_2		\|_{H^s_x	}
		\|			Q_1		\|_{H^s_x	}+
		\|			Q_1		\|_{H^s_x	}+
		\|			Q_2		\|_{H^s_x	}
	\Big\}
	\| \nabla \delta v \|_{L^2_x}^2.
\end{align*}
Hence, assuming $c_{\beta_4}$, $c_{\mu_1}$ and the initial data small enough, we can absorb by the left-hand side the terms related to $\| \delta  \dot Q \|_{L^2_x}$ and $\| \nabla \delta v \|_{L^2_x}$ on the right-hand side, so that the following inequality is fulfilled:
\begin{equation*}
	\frac{\dd}{\dd t}\Psi \lesssim f \Psi.
\end{equation*}
Then, since $\Psi(0)=0$, the Gronwall's inequality yields $\Psi$ to be constantly null, especially
\begin{equation}
	\delta v = v_1-v_2 =0\quad\text{and}\quad \delta Q =Q_1 - Q_2=0.
\end{equation}
This concludes the proof of Theorem~\ref{intro:thm2}. $\Box$
%-------------------------------------------------------------------------------------------------------------------------------------------------
%-------------------------------------------------------------------------------------------------------------------------------------------------
%----------------------------			Twist Waves										------------------------------------------------------------
%-------------------------------------------------------------------------------------------------------------------------------------------------
%-------------------------------------------------------------------------------------------------------------------------------------------------
\section{Twist waves}\label{sec4}

 In the following we consider an example of a ``twist-wave" solution, that is a solution of the coupled system for which the flow $v$ is zero. As noted in the introduction, this amounts to determining a solution of the $Q$-tensor equation \eqref{eqinv:tensors+} with zero flow (hence $v=\Omega=A=0$, $\dot Q$ becomes  $\partial_t Q$) but satisfying {\it an additional nonlinear constraint} namely \eqref{eq:Qtwistconst0}.
 
Our ansatz is inspired from  stationary-case studies \cite{meltinghedgehog}, namely {\it the``melting-hedgehog wave"} obtained by 
taking in $\R^d$ with $d=2$ or $d=3$ the ansatz: 

$$T(t,x):=f(t,|x|)\bar H(x)$$ with $f:\R_+\times \R\to \R$ a function to be determined and $\bar H$ the {\it ``hedgehog"} function (see \cite{meltinghedgehog} for details about its physical significance) :

$$\bar H_{ij}(x):=\frac{x_ix_j}{|x|^2}-\frac{\delta_{ij}}{d}, i,j=1,\dots,d$$ 

 Let us note that in order to avoid a discontinuity at $0$ for $T$ we need to take

\be\label{smelting}
f(t,0)=0,\forall t\ge 0
\ee i.e. the hedgehog ``melts" at the origin.

Then one can check that the equation \eqref{eq:Qwave0} reduces to an equation for $f$ only, namely:

\be\label{eq:s}
Jf_{tt}+\mu_1 f_t=L\left(f_{rr}+f_r\frac{d-1}{r}-\frac{2d}{r^2}f\right)-af+\frac{b(d-2)}{d}f^2-\frac{c(d-1)}{d}f^3
\ee

Note that in order to avoid a singularity at the origin for $f$ we need to further have that 

\be\label{smelting+}
f_r(t,0)=0,\forall t\ge 0
\ee 

Condition \eqref{smelting+} is an apparent singularity, having to do with the radial symmetry, because if one denotes $G(t,x)=f(t,|x|)$ then \eqref{smelting+} holds for $G$ sufficiently smooth in the $x$ variable.

On the other hand in order to check that the constraint equation \eqref{eq:Qtwistconst0} holds it suffices to check that both $\nabla\cdot T_t$ and  $\nabla\cdot (\nabla T\otimes \nabla T)$  can be expressed as gradients \footnote{since the commutator appearing in \eqref{eq:Qtwistconst0} vanishes for our ansatz}.

We have:

\bea\nabla\cdot T_t=\left(f_t(t,|x|)(\frac{x_ix_j}{|x|^2}-\frac{\delta_{ij}}{d})\right)_{,j}=(d-1)\left(\frac{1}{d}\partial_r f_t(t,r)+\frac{f_t(t,r)}{r}\right)\frac{x_i}{|x|}
\eea

Then \eqref{smelting+} implies 
 \be\label{smelting++}
f_{tr}(t,0)=0,\forall t\ge 0
\ee 
 
Thus we have that for sufficiently smooth $f$  there exists a function $g:\R_+\times\R\to \R$ such that $g_r=(d-1)\left(\frac{1}{d}\frac{\partial^2 f}{\partial r\partial t}(t,r)+\frac{ f_t(t,r)}{r}\right)$, hence $\nabla\cdot T_t=\nabla_x g(t,|x|)$.

Furthermore, we have:

$$\nabla\cdot (T\otimes T)=(T_{kl,i}T_{kl,j})_{,j}=T_{kl,ij}T_{kl,i}+T_{kl,i}\Delta T_{kl}$$

As $T_{kl,ij}T_{kl,j}=\frac{1}{2}\left(|\nabla T|\right)_{,i}$ it suffices to check that $T_{kl,i}\Delta T_{kl}$ is a gradient, which in our case amounts to checking that 

$$f_r(f_{rr}+\frac{(d-1)f_r}{r}-\frac{2df}{r^2})\frac{x_i}{|x|}\left(\frac{x_kx_l}{|x|^2}-\frac{\delta_{kl}}{d}\right)\left(\frac{x_kx_l}{|x|^2}-\frac{\delta_{kl}}{d}\right)=\frac{d-1}{d}f_r(f_{rr}+\frac{(d-1)f_r}{r}-\frac{2df}{r^2})\frac{x_i}{|x|}$$ is a gradient.

Thus, assuming \eqref{smelting} and \eqref{smelting+} we have that there exists $h:\R_+\times R\to\R$ such that $h_r=\frac{d-1}{d}f_r(f_{rr}+\frac{(d-1)f_r}{r}-\frac{2df}{r^2})$, hence $\nabla_x h(t,|x|)=\nabla\cdot(\nabla T\otimes \nabla T)$.

\bigskip
These formal computations  provide indeed a twist wave under the smoothness conditions mentioned before. In order to make the above rigorous we continue with the proof of  Proposition~\ref{prop:meltingwave}:

\smallskip
\bproof We proceed in several steps. We will first show that  the $Q$-tensor equation \eqref{eq:Qwave0} has global in time strong solutions for arbitrary initial data. We then  prove a weak-strong uniqueness result for solutions of the $Q$-tensor equation \eqref{eq:Qwave0} and  show that 
the $f$-equation \eqref{eq:s0} has weak solutions and then that these weak solutions provide a   global weak solution of \eqref{eq:Qwave0}. We  use the weak-strong uniqueness to conclude that the solutions provided by $f(t,|x|)\bar H(x)$ are twist-wave solutions.

\bigskip
{\bf Step 1: Global  strong solutions of \eqref{eq:Qwave0}}
\smallskip

We  just provide here the apriori estimates  necessary for obtaining the weak and strong solutions. The actual construction through an  approximation scheme can be done similarly as in the proof of Theorem~\ref{intro:thm2}.

We first obtain the apriori boundedness of the $L^2$ norm. To this end we multiply \eqref{eq:Qwave0} by $T_t$, integrate over $\R^d$ and by parts\footnote{ throughout the proof we always assume as usually that we can integrate by parts without boundary terms}, to get:

\bea
\frac{d}{dt}\int_{\R^d}& \frac{J}{2}|T_t|^2+\frac{L}{2}|\nabla T|^2+\psi_B(T)\,dx+\mu_1\int_{\R^d} |T_t|^2\,dx=0
\label{eq:Qt}
\eea

We note that because $\psi_B(T)$ can be negative this does not suffice for obtaining estimates on  the $L^2$ norm of $T$. Thus we multiply \eqref{eq:Qwave0} by $T$, integrate over $\R^d$ and by parts, to get:

\bea\label{eq:QQ}
\frac{J}{2}\frac{d^2}{dt^2}\int_{\R^d}  |T|^2\,dx&-J\int_{\R^d} |T_t|^2\,dx+\frac{\mu_1}{2}\frac{d}{dt}\int_{\R^d} |T|^2\,dx+L\int_{\R^d} |\nabla T|^2\,dx=\non\\
&\int_{\R^d} (-a|T|^2+b\textrm{tr}(T^3)-c|T|^4)\,dx\eea

We multiply \eqref{eq:Qt} by $J$ and \eqref{eq:QQ} by $\mu_1$ and add them together to get:

\bea
\frac{J\mu_1}{2}\frac{d^2}{dt^2}\int_{\R^d}|T|^2\,dx+\frac{d}{dt}\int_{\R^d} \frac{J^2}{2}|T_t|^2&+\frac{LJ}{2}|\nabla T|^2+J\psi_B(T)+\frac{\mu_1^2}{2}|T|^2\,dx\non\\
&+L\mu_1\int_{\R^d} |\nabla T|^2\,dx\le C\int_{\R^d} | T|^2\,dx
\eea with the large enough constant $C$ depending just on $\mu_1$, $a$, $b$ and $c$.

Integrating over $[0,t]$ we obtain:

\bea
\frac{J\mu_1}{2}&\frac{d}{dt}\int_{\R^d}|T|^2(t,x)\,dx+\int_{\R^d} \left(\frac{J^2}{2}|T_t|^2+\frac{LJ}{2}|\nabla T|^2+J\psi_B(T)+\frac{\mu_1^2}{2}|T|^2 \right)(t,x)\,dx\le \non\\
&\le J\mu_1\int_{\R^d}(T_t:T)(0,x)\,dx+\int_{\R^d} \left(\frac{J^2}{2}|T_t|^2+\frac{LJ}{2}|\nabla T|^2+J\psi_B(T)+\frac{\mu_1^2}{2}|T|^2\right)(0,x)\,dx\non\\
&+C\int_0^t \int_{\R^d} |T|^2(s,x)\,ds\,dx
\eea which implies:

\be
\frac{J\mu_1}{2}\frac{d}{dt}\int_{\R^d}|T|^2(t,x)\,dx+\int_{\R^d} \frac{J^2}{2}|T_t|^2(t,x)\,dx\le C+C\int_{\R^d} |T(t,x)|^2\,dx+\int_0^t \int_{\R^d} |T|^2(s,x)\,ds\,dx
\ee with the constants $C$ depending only on the initial data and the constants in the equation.  Thus using  Gronwall inequality and Fubini (to turn the double time integral into a weighted time integral)  together with \eqref{eq:Qt} we obtain apriori control over certain energy-level norms:

\be\label{eq:weakQ}
T\in L^\infty(0,T;L^2)\cap L^\infty(0,T;H^1)\cap L^\infty(0,T;L^4)<C(a,b,c,J,\mu_1,T,\|T_0\|_{H^1},\|\partial_t T_0\|_{L^2})
\ee

\be\label{eq:weakQt}
T_t\in L^\infty(0,T;L^2)<C(a,b,c,J,\mu_1,T,\|T_0\|_{L^2},\|\partial_t T_0\|_{L^2})
\ee
 
In order to obtain control over the $H^s$ norm we multiply \eqref{eq:Qwave0} with $T_t$ in the $H^s$ inner product, with $s>\frac{d}{2}$ obtaining:

\bea
\frac{1}{2}\frac{d}{dt}\left(J\|T_t\|_{H^s}^2+L\|\nabla T\|_{H^s}^2\right)+\mu_1\| T_t\|_{H^s}^2\,dx&=\langle T_t, -aT+b(T^2-\frac{|T|^2}{d}I_d)-cT|T|^2\rangle_{H^s}\non\\
&\le C\|T_t\|_{H^s}\left(\|T\|_{H^s}+\|T\|_{H^s}^3\right)\non\\
&\le C\|T_t\|_{H^s}^2+C\|T\|_{H^s}^2+\varepsilon\|T\|_{H^s}^6\non\\
&\le  C\|T_t\|_{H^s}^2+C\|T\|_{H^s}^2+\varepsilon \|T\|_{H^1}^{\frac{6}{s}}\cdot\varepsilon \|T\|_{H^{s+1}}^{\frac{6(s-1)}{s}}
\eea where in the last inequality we estimated  $H^s$ through interpolation between $H^1$ and $H^{s+1}$.
In order to be able to control the term $\varepsilon \|T\|_{H^s}^{\frac{6(s-1)}{s}}$ we need\footnote{ In here we need $d=2$ as  we assumed that $\frac{d}{2}<s$ and $d=3$ would contradict the restriction $s\le \frac{3}{2}$}
 to have $\frac{6(s-1)}{s}\le 2$ i.e. $s\le\frac{3}{2}$. Thus using Gronwal the previous lemma gives apriori control on $\|T_t\|_{H^s}^2+\|T\|_{H^{s+1}}^2$ in $L^\infty(0,T)$ for any $T>0$, provided that $s\le\frac{3}{2}$.  We can then repeate the same estimates as above but at the last line estimate through interpolation of $H^s$ between $H^{\frac{3}{2}}$ and $H^{s+1}$ with $s\le \frac{5}{2}$. Repeating inductively we obtain control for arbitrary $s>\frac{d}{2}$.

\bigskip 
{\bf Step 2:  Weak-strong uniqueness}
\smallskip

We assume that the weak solutions have regularity:

\be
T\in L^\infty(0,T;L^2)\cap L^\infty(0,T;H^1)\cap L^\infty(0,T;L^4)
\ee

\be
T_t\in L^\infty(0,T;L^2)
\ee

We consider the difference of two solutions $T_1$ and $T_2$ of \eqref{eq:Qwave0} with $T_1$ being a weak solution and $T_2$ a strong solution. We denote $\dT:=T_1-T_2$, and note that 

$$\dT(0,x)=\partial_t \dT(0,x)\equiv 0$$ and $\dT$ satisfies the equation:

\bea
J\dT_{tt}+\mu_1 \dT_t=&L\Delta \dT-a\dT+b\left(\dT^2+T_2\dT+\dT T_2-\frac{|\dT|^2}{d}I_d-\frac{2T_2:\dT}{d}I_d\right)\non\\
&-c\left[ \dT|\dT|^2+\dT|T_2|^2+2\dT(\dT:T_2)+T_2|\dT|^2+2T_2(\dT:T_2)\right]\label{eq:dq}
\eea

We multiply the last relation by $\dT_t$, integrate over $\R^d$ and by parts, to get:

\bea
\frac{d}{dt}\int_{\R^d}& \frac{J}{2}|\delta T_t|^2+\frac{L}{2}|\nabla \dT|^2+\psi_B(\dT)\,dx+\mu_1\int_{\R^d} |\dT_t|^2\,dx=b\int_{\R^d} (T_2\dT+\dT T_2):\dT_t\,dx\non\\
&-c\int_{\R^d} 
|T_2|^2(\dT:\dT_t)+2(\dT:T_2)(\dT:\dT_t)+|\dT|^2(T_2:\dT_t)+2(\dT:T_2)(\dT_t:T_2)\,dx \non\\
&\le C\left(\int_{\R^d} |\dT|^2+|\dT|^4\,dx\right)^{\frac{1}{2}}\left(\int_{\R^d} |\dT_t|^2\,dx\right)^{\frac{1}{2}}
\label{eq:dQQt}
\eea with $C$ a constant depending on $\|T_2\|_{L^\infty}$, $b$ and $c$.

On the other hand, multiplying \eqref{eq:dq} by $\dT$, integrating over $\R^d$ and by parts, we get:

\bea
\frac{J}{2}\frac{d^2}{dt^2}\int_{\R^d}  |\dT|^2\,dx&-J\int_{\R^d} |\dT_t|^2\,dx+\frac{\mu_1}{2}\frac{d}{dt}\int_{\R^d} |\dT|^2\,dx+L\int_{\R^d} |\nabla \dT|^2\,dx=\non\\
&\int_{\R^d} (-a|\dT|^2+b\textrm{tr}(\dT^3)-c|\dT|^4)\,dx+b\int_{\R^d} (T_2\dT+\dT T_2):\dT\,dx\non\\
&-c\int_{\R^d} 
\left(|\dT|^2|T_2|^2+3(\dT:T_2)|\dT|^2+2(\dT:T_2)^2\right)\,dx\non\\
& \le C\int_{\R^d} |\dT|^2-\frac{c}{2}\int_{\R^d}|\dT|^4\,dx\,dx\label{eq:dQQ}\eea with $C$  a constant depending on $\|T_2\|_{L^\infty}$, $a$, $b$ and $c$.

We multiply \eqref{eq:dQQt} by $J$ and \eqref{eq:dQQ} by $\mu_1$ and add them together to get:

\bea
\frac{J\mu_1}{2}\frac{d^2}{dt^2}\int_{\R^d}|\dT|^2\,dx+\frac{d}{dt}\int_{\R^d} \frac{J^2}{2}|\dT_t|^2+\frac{LJ}{2}|\nabla\dT|^2+J\psi_B(\dT)&+\frac{\mu_1^2}{2}|\dT|^2\,dx+L\mu_1\int_{\R^d} |\nabla\dT|^2\,dx\non\\
&\le C\int_{\R^d} |\dT_t|^2+|\dT|^2\,dx
\eea with $C$  a constant depending on $\|T_2\|_{L^\infty}$, $a$, $b$ and $c$.

Integrating over $[0,t]$ we obtain:

\bea
\frac{J\mu_1}{2}&\frac{d}{dt}\int_{\R^d}|\dT|^2(t,x)\,dx+\int_{\R^d} \left(\frac{J^2}{2}|\dT_t|^2+\frac{LJ}{2}|\nabla\dT|^2+J\psi_B(\dT)+\frac{\mu_1^2}{2}|\dT|^2 \right)(t,x)\,dx\le \non\\
&\le J\mu_1\int_{\R^d}\underbrace{(\dT_t:\dT)(0,x)}_{=0}\,dx+\int_{\R^d} \left(\underbrace{\frac{J^2}{2}|\dT_t|^2+\frac{LJ}{2}|\nabla\dT|^2+J\psi_B(\dT)+\frac{\mu_1^2}{2}|\dT|^2}_{=0}\right)(0,x)\,dx\non\\
&+C\int_0^t \int_{\R^d} (|\dT_t|^2+|\dT|^2)(s,x)\,ds\,dx
\eea which implies:

\bea
\frac{J\mu_1}{2}\frac{d}{dt}\int_{\R^d}|\dT|^2(t,x)\,dx&+\int_{\R^d} \frac{J^2}{2}|\dT_t|^2(t,x)\,dx\non\\
&\le C\int_{\R^d} |\dT(t,x)|^2\,dx+C\int_0^t \int_{\R^d} (|\dT_t|^2+|\dT|^2)(s,x)\,ds\,dx
\eea 

Integrating one more time and using $\dT_t(0,\cdot)=\dT(0,\cdot)\equiv 0$ and Gronwal we get that $\dT_t(t,\cdot)=\dT(t,\cdot)\equiv 0$ for all $t>0$.

\bigskip
{\bf Step 3: weak solutions of the $f$-equation \eqref{eq:s}}
\smallskip

We will just provide here the apriori estimates  necessary for obtaining the weak solutions. The actual construction of the approximation scheme can be done through a straightforward modification of the scheme used in the proof of Theorem~\ref{intro:thm2} and it is left to the interested reader.

We assume that \eqref{eq:s} has a classical solution. We multiply it by $f_tr^2$ and integrate over $\R$ to get

\be\label{est:st}
\frac{d}{dt}\int_\R \left\{ \frac{J}{2} f_t^2+\frac{L}{2}f_r^2+h_B(f)+d\frac{f^2}{r^2}\right\} r^2\,dr+\mu_1\int_\R f_t^2 r^2\,dr=0
\ee where we used the reduced potential

\be\label{def:hb}
h_B(f)=\frac{a}{2}f^2-\frac{b(d-2)}{3d}f^3+\frac{c(d-1)}{4d}f^4
\ee

On the other hand, multiplying \eqref{eq:s} by $fr^2$ and integrating over $\R$ we obtain:

\bea\label{est:s}
\frac{d^2}{dt^2}\int_\R \frac{J}{2}f^2r^2\,dr-J\int_\R f_t^2r^2\,dr+\frac{\mu_1}{2}\frac{d}{dt}\int_\R f^2r^2\,dr+L\int_\R f_r^2r^2\,dr=\non\\
-2d\int_\R f^2\,dr+\int_\R (-af^2+\frac{b(d-2)}{d}f^3-\frac{c(d-1)}{d}f^4)r^2\,dr
\eea

We now multiply \eqref{est:st} by $J$ and add to it \eqref{est:s} multiplied by $\mu_1$ to get:

\bea
\frac{d^2}{dt^2}\int_\R \frac{J\mu_1}{2}f^2r^2\,dr+\frac{d}{dt}\int_\R J\left\{ \frac{J}{2} f_t^2+\frac{L}{2}f_r^2+h_B(f)+d\frac{f^2}{r^2}\right\} r^2\,dr+\frac{\mu_1^2}{2}\frac{d}{dt}\int_\R f^2r^2\,dr=\non\\
-L\mu_1\int_\R f_r^2r^2\,dr-2d\mu_1\int_\R f^2\,dr+\int_\R \mu_1(-af^2+\frac{b(d-2)}{d}f^3-\frac{c(d-1)}{d}f^4)r^2\,dr
\eea

Integrating over $[0,t]$ we get:

\bea
\frac{d}{dt}\int_\R \frac{J\mu_1}{2}f^2(t,r)\,r^2\,dr+\int_\R J\left\{ \frac{J}{2} f_t^2+\frac{L}{2}f_r^2+h_B(f)+d\frac{f^2}{r^2}\right\}(t,r) r^2\,dr+\frac{\mu_1^2}{2}\int_\R f^2(t,r)r^2dr=\non\\
\int_R J\mu_1 f_t(0,r)f(0,r)\,r^2dr+\int_\R J\left\{ \frac{J}{2} f_t^2+\frac{L}{2}f_r^2+h_B(f)+d\frac{f^2}{r^2}\right\}(0,r) r^2\,dr+\frac{\mu_1^2}{2}\int_\R f^2(0,r)r^2dr\non\\
-\mu_1\int_0^t\int_\R \left(Lf_r^2+2d f^2+af^2-\frac{b(d-2)}{d}f^3+\frac{c(d-1)}{d}f^4)\right)(s,r)r^2\,drds
\eea

Integrating one more time over $[0,t]$ we further obtain:
\bea
\int_\R \frac{J\mu_1}{2}f^2(t,r)\,r^2\,dr+\int_0^t\int_\R \left\{ J\left(\frac{J}{2} f_t^2+\frac{L}{2}f_r^2+h_B(f)+d\frac{f^2}{r^2}\right)+\frac{\mu_1^2}{2} f^2\right\}(s,r)r^2drds=\non\\
\int_\R \frac{J\mu_1}{2}f^2(0,r)\,r^2\,dr+t\int_R J\mu_1 f_t(0,r)f(0,r)\,r^2dr+t\int_\R J\left\{ \frac{J}{2} f_t^2+\frac{L}{2}f_r^2+h_B(f)+d\frac{f^2}{r^2}\right\}(0,r) r^2\,dr\non\\
+t\frac{\mu_1^2}{2}\int_\R f^2(0,r)r^2dr-\mu_1\int_0^t\int_0^\tau\int_\R \left(Lf_r^2+2df^2+af^2-\frac{b(d-2)}{d}f^3+\frac{c(d-1)}{d}f^4)\right)(s,r)r^2\,drds\,d\tau
\eea

Using the fact that for an arbitrary  function $a\in L^1_{loc}(\R;\R)$ we have: $\int_0^t\int_0^\tau a(s)\,ds\,d\tau=\int_0^t (t-\tau)a(\tau)\,d\tau$ and that $J,\mu_1,L,c>0$ we obtain out of the last relation:

\be
\int_\R \frac{J\mu_1}{2}f^2(t,r)\,r^2\,dr\le C_1+C_2t+C_3t\int_0^t\int_{\R} f^2(s,r)\,r^2\,drds
\ee which implies for any $T>0$ that $f\in L^\infty(0,T; L^2(\R,r^2\,dr))$. Using this bound and integrating \eqref{est:st} on $[0,T]$ we also get $f\in L^\infty (0,T;H^1(\R,r^2\,dr))\cap L^\infty(0,T;L^4(\R,r^2\,dr))$ and $f_t\in L^\infty (0,T;L^2(\R,r^2\,dr))$.

\bigskip
 {\bf Step 4: the existence of smooth twist solutions}
\smallskip
One can easily see that the  previously obtained weak solution of equation of the $f$-equation \eqref{eq:s0} will provide a weak solution of the $Q$-equation  \eqref{eq:Qwave0} through the formula $T(t,x)=f(t,x)\bar H(x), \forall t\ge 0,x\in \R^d$.
Due to the weak-strong uniqueness we have that $T(t,x)$ is also a strong solutions and thus since we can take $s$ arbitrary we have $T$  smooth. In particular $T$ is continuous at $0$ which, because of the discontinuity of $\bar H$ necessarily implies  that $f(t,\cdot)$ is continuous on $[0,\infty)$ for any $t\ge 0$ and  $f(t,0)=0$ hence condition \eqref{smelting} holds. Furthermore by evaluating the representation formula $T(t,x)=f(t,|x|)\bar H$ for at the component $11$ of the matrices and at the point $x=(0,r)$ we 
have that $T_{11}(t,0,r)=f(t,r)(\frac{r^2}{r^2}-\frac{1}{3})$ hence $f(t,r)=\frac{3}{2}T_{11}(t,0,r)$. Similarly $f(t,r)=\frac{3}{2}T_{11}(t,0,-r)$. Since $T_{11}$ is a smooth function, we have that its restriction to the line $\{(0,r),r\in \mathbb{R}\}$ is also a smooth function that is furthermore even. Thus we get that $f\in C^\infty[0,\infty)$ with $f_r(t,0)=0$ as a consquence of the evenness of $T_{11}(t,0,\cdot)$. A similar argument holds for $f_t$ providing its smoothness and $f_{rt}(t,0)=0$, hence conditions \eqref{smelting+} and \eqref{smelting++} hold. 
Thus, the  arguments provided before the statement of Proposition~\ref{prop:meltingwave} hold in a rigorous sense and we have obtained a twist wave.

\eproof

\section{Some Open Problems}
\label{sec:openprob} The study and the techniques developed in the paper generate some natural open questions:
\begin{itemize}
\item {\bf Global existence for small data and $a$ negative} The main technical assumption that we make for obtaining the existence of a global solution is that the coefficient  $a$ appearing in equation \eqref{eqinv:tensors+} is positive. This captures a physically relevant regime, in which the nematic state is nevertheless just a local but not global minimizer of the bulk potential. It would be interesting and technically challenging to see if one can obtain global existence for $a$ negative. 

We suspect that the case $a$ negative should be treated with different tools, as that case allows in the case $J=0$ for a solution of the $Q$-equation (without flow) whose $L^p$ norms increase (which seems incompatible with our strategy of the proof of global existence).

\smallskip
\item {\bf Long-time behaviour, and relation to inertieless version}

The usual expectation for the case of the damped wave equation is that in the long time one has a diffusive-type behaviour and this is shown in a number of papers (see for instance \cite{karch}). Our system has a resemblance with a damped wave equation but it is not clear if its structure allows for a similar conclusion.

\smallskip
\item {\bf The $J\to 0$ limit} It is a natural question to try to understand the singular limit $J\to 0$ in order to understand the effect that taking $J=0$ has. It is natural to conjecture that after an initial boundary layer in time the solutions will converge strongly to the formal $J=0$ limit.

\smallskip
\item  {\bf More twist-waves (with genuine wave structure)} We provide in the last section an example of a twist-wave solution that will remain a solution even if one  also sets $J=0$. It would be interesting to provide more such examples, and in particular to understand if there are examples which would not survive as twist-wave solutions when setting formally $J=0$.

\smallskip
\item {\bf The stability of the twist wave solution} In  \cite{meltinghedgehog} it was shown that in the stationary case the melting hedgehog solution is stable. This is a crucial feature for determining the physical relevance of such a solution because only the stable solutions can be observed experimentally. It would be thus very interesting to see if one has dynamical stability of the ``melting hedgehog wave" solution, i.e. if one starts with the an approximation of the hedgehog initial data in $Q$ and a small initial data in $u$ will this solution stay close to the melting hedgehog one? Will it evolve in the long-time to the melting hedgehog wave? 
\end{itemize}

\bigskip
\section*{Appendix}
\begin{lemma}\label{appx:lemma1}
	Let $y$ be a positive function in $W^{1,1}_{loc}(\RR_+)$ and $x$ a function in $L^1_{loc}(\RR_+)$ that is  almost everywhere positive.  Let us assume that 
	\begin{equation}\label{appx_lemma_ineq}
		y'(t)+ x(t) \leq C y(t)x(t),
	\end{equation}
	for almost every $t$ in $\RR_+$.
	
	There exists $\varepsilon_0>0$ a suitably small number such that if we take the initial datum $y(0)=y_0\in (0,\varepsilon_0)$, then $y$ and $x$ belong to $L^\infty (\RR_+)$ and 
	$L^1(\RR_+)$ respectively, and moreover
	\begin{equation*}
		\|y\|_{L^\infty(\RR_+)} +\|x\|_{L^1(\RR_+)} \leq y_0.
	\end{equation*}
\end{lemma}
\begin{proof}
	Assuming $y_0\leq 1/4C$, we define $T>0$ as the $\sup $ of $t>0$ such that $y(t)< 1/2C$. Then, for every $t\in [0,T]$ we get
	\begin{equation*}
		y'(t) + \frac 12 x(t) \leq 0,
	\end{equation*}
	so that, integrating from $0$ to $T$, we deduce
	\begin{equation*}
		y(T) + \frac 12\int_0^T x(t) \leq y_0<\frac {1}{2C}.
	\end{equation*}
	This yields that $T=+\infty$ and that
	\begin{equation*}
		\|y\|_{L^\infty(\RR_+)} +\|x\|_{L^1(\RR_+)} \leq y_0.
	\end{equation*}
\end{proof}

%

%\frac{\mu_1}{4}	\frac{\dd}{\dd t}	\|	Q	\|_{H^s}^2

%\eqref{eqinv:tensors+}

% Please add your grant numbers here if applicable.
%\thanks{Grants acknowledgements.}

%--- Table of contents followed by a page break is useful while editing,
%--- especially with active links
%\ifdraft
%	\tableofcontents
%	\newpage
%\fi

%
%----------------------------------------------
%
%}}}
\section*{Acknowledgement}
The activity of Francesco de Anna on this work was partially supported by  funding from the European Union's
Horizon 2020 research and innovation programme under the Marie Sklodowska-Curie grant
agreement ModCompShock, No 642768, through a three-month fellowship at the University of Sussex.

The activity of Arghir Zarnescu on this work was partially
supported by a grant of the Romanian National Authority for Scientific Research and Innovation,
CNCS-UEFISCDI, project number PN-II-RU-TE-2014-4-0657 ; by the
Project of the Spanish Ministry of Economy and Competitiveness with reference
MTM2013-40824-P; by the Basque Government through the BERC
2014-2017 program; and by the Spanish Ministry of Economy and Competitiveness
MINECO: BCAM Severo Ochoa accreditation SEV-2013-0323.

Both authors thank Professor Marius Paicu for inspiring discussions  and also to the anonymous referee for the careful reading and a number of remarks that helped to improve the paper.

%}}}1

\begin{bibdiv}%{{{1
\begin{biblist}

%\bib{BallMajumdar}{article}{%{{{
%      author={Ball, J.M.},
%   author={Majumdar, A.},
%    title={Nematic Liquid Crystals : from Maier-Saupe to a Continuum Theory},
%     journal={Mol. Cryst. Liq. Cryst},
%   volume={525},
%   date={2010},
%   pages={1-11},
%}
\bib{Ericksentwist}{article}{%{{{
   author={Ericksen, J.L.},
    title={Twist waves in liquid crystals},
   journal={Q.J. Mech. Appl. Math.},
   volume={21},
   date={1968},
   pages={463},
}

\bib{commutator}{article}{
  title={Higher order commutator estimates and local existence for the non-resistive MHD equations and related models},
  author={Fefferman, Charles L },
  author={McCormick, David S},
  author={Robinson, James C},
  author={Rodrigo, Jose L},
  journal={Journal of Functional Analysis},
  volume={267},
  number={4},
  pages={1035--1056},
  year={2014},
  publisher={Elsevier}
}

\bib{meltinghedgehog}{article}{
 author={Ignat, Radu},
 author={Nguyen, Luc},
 author={Slastikov, Valeriy},
 author={Zarnescu, Arghir},
 title={Stability of the melting hedgehog in the Landau--de Gennes theory of nematic liquid crystals},
  journal={Archive for Rational Mechanics and Analysis},
  volume={215},
  number={2},
  pages={633--673},
  year={2015},
  }
  
  
\bib{karch}{article}{
  title={Selfsimilar profiles in large time asymptotics of solutions to damped wave equations},
  author={Karch, Grzegorz},
  journal={Studia Mathematica},
  volume={143},
  number={2},
  pages={175--197},
  year={2000},
  publisher={Institute of Mathematics Polish Academy of Sciences}
}
\bib{KWZ}{article}{
  title={Dynamic Statistical Scaling in the Landau--de Gennes Theory of Nematic Liquid Crystals},
  author={Kirr, E.},
  author={Wilkinson, M.},
  author={Zarnescu, A.},
  journal={Journal of Statistical Physics},
  volume={155},
  number={4},
  pages={625--657},
  year={2014},
  publisher={Springer}
}
\bib{Leslie}{article}{
    AUTHOR = {Leslie, F. M.},
     TITLE = {Some constitutive equations for liquid crystals},
   JOURNAL = {Arch. Rational Mech. Anal.},
  FJOURNAL = {Archive for Rational Mechanics and Analysis},
    VOLUME = {28},
      YEAR = {1968},
    NUMBER = {4},
     PAGES = {265--283},
      ISSN = {0003-9527},
     CODEN = {AVRMAW},
   MRCLASS = {Contributed Item},
  MRNUMBER = {1553506},
       DOI = {10.1007/BF00251810},
       URL = {http://dx.doi.org/10.1007/BF00251810},
}

\bib{chun}{article}{
  title={Existence of Solutions for the Ericksen-Leslie System},
  author={Lin, F.-H. and Liu, C.},
  journal={Archive for Rational Mechanics and Analysis},
  volume={154},
  number={2},
  pages={135--156},
  year={2000},
  publisher={Springer}
}



\bib{leslieacceleration}{article}{
  title={Theory of flow phenomena in liquid crystals},
  author={Leslie, F. M.},
  journal={Advances in liquid crystals},
  volume={4},
  pages={1--81},
  year={1979},
  publisher={Academic Press New York}
}  


 \bib{NewtonMottram}{article}{
  title={Introduction to Q-tensor theory},
  author={Mottram, Nigel J}
  author= {Newton, Christopher JP},
  journal={arXiv preprint arXiv:1409.3542},
  year={2014}
}

\bib{PZ1}{article}{
  title={Energy dissipation and regularity for a coupled Navier--Stokes and Q-tensor system},
  author={Paicu, Marius}
  author={Zarnescu, Arghir},
  journal={Archive for Rational Mechanics and Analysis},
  volume={203},
  number={1},
  pages={45--67},
  year={2012},
  publisher={Springer}
}

\bib{popa2003waves}{article}{,
  title={Waves at the nematic-isotropic interface: The role of surface tension anisotropy, curvature elasticity, and backflow effects},
  author={Popa-Nita, V}
  author={Oswald, P},
  journal={Physical Review E},
  volume={68},
  number={6},
  pages={061707},
  year={2003},
  publisher={APS}
}
\bib{popa2005waves}{article}{
  title={Waves at the nematic-isotropic interface: Thermotropic nematogen--non-nematogen mixtures},
  author={Popa-Nita, V},
  author={Sluckin, TJ}, 
  author={Kralj, Samo},
  journal={Physical Review E},
  volume={71},
  number={6},
  pages={061706},
  year={2005},
  publisher={APS}
}

\bib{QianSheng}{article}{%{{{
   author={Qian, T.},
   author={Sheng, P.},
   title={Generalized hydrodynamic equations for nematic liquid crystals},
   journal={Phys. Rev. E},
   volume={58},
   date={1998},
   pages={7475--7485},
}
\bib{SvensekZummer}{article}{%{{{
   author={Svensek, D.},
   author={Zummer, S.},
   title={Hydrodynamics of pair-annihilating disclination lines in nematic liquid crystals},
   journal={Phys. Rev. E},
   volume={66},
   date={2002},
   pages={021712},
}



\bib{Wilkinson}{article}{
  title={Strictly Physical Global Weak Solutions of a Navier--Stokes Q-tensor System with Singular Potential},
  author={Wilkinson, Mark},
  journal={Archive for Rational Mechanics and Analysis},
  volume={218},
  number={1},
  pages={487--526},
  year={2015},
  publisher={Springer}
}

\bib{Xiaoglobal}{article}{
  title={Global solution to the three-dimensional liquid crystal flows of Q-tensor model},
  author={Xiao, Yao},
  journal={arXiv preprint arXiv:1601.07695},
  year={2016},
}
\bib{ZarIntro}{article}{%{{{
  title={Topics in the Q-tensor theory of liquid crystals},
  author={Zarnescu, Arghir},
  journal={Topics in mathematical modeling and analysis},
  volume={7},
  pages={187--252},
  year={2012},
}%}}}
\end{biblist}
\end{bibdiv}
\end{document}